\documentclass[10pt]{article}
\usepackage{lmodern}
\usepackage{amsmath}
\usepackage[T1]{fontenc}
\usepackage[utf8]{inputenc}
\usepackage{authblk}
\usepackage{amsfonts}
\usepackage{graphicx}
\usepackage{rotating}
\usepackage{amssymb}
\usepackage[english]{babel}
\usepackage{xcolor}
\usepackage{amsthm}
\usepackage{graphicx}
\usepackage{mathrsfs}
\usepackage{makecell}
\usepackage{microtype}
\usepackage{mathscinet}
\usepackage{array}
\usepackage{multirow}
\usepackage{enumerate}
\usepackage[cal=boondoxo,bb=ams]{mathalfa}
\usepackage{hyperref}
\hypersetup{hidelinks}
\usepackage{booktabs}

\newtheorem{defn}{Definition}[section]
\newtheorem{theorem}{Theorem}[section]
\newtheorem{prop}{Proposition}[section]

\newtheorem{remark}{Remark}[section]

\newcommand{\ml}{\mathcal}
\newcommand{\mb}{\mathbb}

\DeclareMathOperator{\lin}{lin}
\DeclareMathOperator{\nlin}{nlin}

\def\XXint#1#2#3{{\setbox0=\hbox{$#1{#2#3}{\int}$ }
		\vcenter{\hbox{$#2#3$ }}\kern-.6\wd0}}

\title{Global in-time existence of solutions for the complex-valued Jordan-Moore-Gibson-Thompson equations of Westervelt-type under different conditions on initial data}
\author[1]{Wenhui Chen\thanks{Wenhui Chen (wenhui.chen.math@gmail.com)}}
\affil[1]{School of Mathematics and Information Science, Guangzhou University,\authorcr 510006 Guangzhou, P.R. China}

\date{}

\begin{document}
		\maketitle

		\begin{abstract}
			\medskip
We are interested in the global in-time existence of solutions for the complex-valued Jordan-Moore-Gibson-Thompson (JMGT) equations of Westervelt-type, namely,
\begin{align*}
\tau\partial_t^3\psi+\partial_t^2\psi+\mathcal{A}\psi+(\delta+\tau)\mathcal{A}\partial_t\psi=(1+\tfrac{B}{2A})\partial_t[(\partial_t\psi)^2]
\end{align*}
in the whole space $\mb{R}^n$, with $\tau,\delta,\frac{B}{A}\in\mb{R}_+$ and the fractional Laplacian $\ml{A}:=(-\Delta)^{\sigma}$ equipping $\sigma\in\mb{R}_+$. Our aims are twofold. For one thing, by considering the rough initial data with their Fourier support restrictions in a suitable subset of first octant, we demonstrate a global in-time existence result without requiring the smallness of initial data. For another, by removing these Fourier support restrictions, we prove another global in-time existence result for the equivalent strongly coupled JMGT systems, where the real and imaginary parts of initial data, respectively, belong to regular Sobolev spaces with different additional Lebesgue integrabilities.
			\\
			
			
\noindent\textbf{AMS Classification (2020)} 35L76, 35L30, 35A01, 42B35, 35R11
		\end{abstract}
\fontsize{12}{15}
\selectfont

\section{Introduction}\setcounter{equation}{0}\label{Section-Introduction}
\hspace{5mm}In this paper, we consider the following complex-valued Jordan-Moore-Gibson-Thompson (JMGT) equations of Westervelt-type, arising from the nonlinear acoustics in thermoviscous fluids (that will be introduced in Subsection \ref{Sub-Section-Real-Value-JMGT}):
\begin{align}\label{Eq-Complex-JGMT-W}
\begin{cases}
\tau\partial_t^3\psi+\partial_t^2\psi+\mathcal{A}\psi+(\delta+\tau)\mathcal{A}\partial_t\psi=(1+\frac{B}{2A})\partial_t[(\partial_t\psi)^2],&x\in\mb{R}^n,\ t\in\mb{R}_+,\\
\psi(0,x)=\psi_0(x),\ \partial_t\psi(0,x)=\psi_1(x),\ \partial_t^2\psi(0,x)=\psi_2(x),&x\in\mb{R}^n,
\end{cases}
\end{align}
with the thermal relaxation $\tau\in\mb{R}_+$, the diffusivity of sound $\delta\in\mb{R}_+$, the constant ratio $\frac{B}{A}\in\mb{R}_+$ for nonlinearities in a given flow, and the fractional operator $\ml{A}:=(-\Delta)^{\sigma}$ carrying $\sigma\in\mb{R}_+$, where the complex-valued scalar generalized function $\psi=\psi(t,x)\in\mb{C}$ stands for the acoustic velocity potential. In the above statement, the fractional Laplacian $\ml{A}$ is defined via
\begin{align*}
	\langle \ml{A}f,\phi\rangle=\langle f,\ml{A}\phi\rangle\  \mbox{for}\  f\in\ml{S}'\  \mbox{and}\  \phi\in\ml{S}
\end{align*}
with the Schwartz space $\ml{S}$ and its dual space $\ml{S}'$, where $\ml{A}\phi=(-\Delta)^{\sigma}\phi:=\ml{F}^{-1}(|\xi|^{2\sigma}\widehat{\phi}(\xi))$ with $\sigma\in\mb{R}_+$. Note that the Fourier transform and its inverse transform are, respectively, written by $\ml{F}$ and $\ml{F}^{-1}$. This manuscript mainly studies the global in-time existence of solutions for the JMGT equations \eqref{Eq-Complex-JGMT-W} with
\begin{itemize}
	\item  the initial data belonging to rough spaces without requiring their small size (i.e. the so-called rough large data) whose Fourier supports localize in a suitable subset of first octant;
	\item  the initial data belonging to regular spaces requiring their small size (i.e. the so-called regular small data) but without Fourier support restrictions.
\end{itemize}
In other words, we are interested in some influence from different conditions of initial data on the global in-time solutions for the complex-valued nonlinear acoustics models \eqref{Eq-Complex-JGMT-W}.

 \subsection{Function spaces and distribution spaces}\label{Sub-Section-Function-space}
 \hspace{5mm}In order to clearly introduce the background of JMGT equations \eqref{Eq-Complex-JGMT-W} as well as clarify our motivations, in this subsection, let us firstly recall some spaces for initial data and solutions that will be used frequently.
 
 As usual, $L^m$ is referred to the Lebesgue space (integrability) endowed with its norm $\|f\|_{L^m}$ for $m\in[1,+\infty]$. The Lebesgue mixed space $L_t^{\gamma}L_x^m$ endows with its norm
 \begin{align*}
 	\|g\|_{L^{\gamma}_tL^m_x}:=\big\|\|g(t,\cdot)\|_{L^m_x}\big\|_{L^{\gamma}_t}
 \end{align*}
 carrying $\gamma,m\in[1,+\infty]$. The homogeneous [resp. inhomogeneous] Sobolev space is defined via $\dot{H}^s:=|D|^{-s}L^2$ [resp. $H^s:=\langle D\rangle^{-s}L^2$] with any $s\in\mb{R}$. The symbols of differential operators $|D|^s$ and $\langle D\rangle^s$ are, respectively, $|\xi|^s$ and $\langle \xi\rangle^s$ carrying the Japanese bracket $\langle\xi\rangle:=\sqrt{1+|\xi|^2}$ for $s\in\mb{R}$.
 
 Let us recall the Gelfand-Shilov space of Beurling-type $\ml{S}_{\mathrm{G-S}}$ (cf. \cite{Gelfand-Shilov=1968}) that is a Fr\'echet space defined via
 \begin{align*}
 	\ml{S}_{\mathrm{G-S}}:=\left\{f\in\ml{S}: \sup\limits_{x\in\mb{R}^n}\big(\mathrm{e}^{\lambda|x|}|f(x)|\big)+\sup\limits_{\xi\in\mb{R}^n}\big(\mathrm{e}^{\lambda|\xi|}|\widehat{f}(\xi)|\big)<+\infty \ \mbox{for any} \ \lambda>0  \right\}.
 \end{align*}
 The Fourier transform on $\ml{S}_{\mathrm{G-S}}'$, i.e. the dual space of $\ml{S}_{\mathrm{G-S}}$, can be defined by its duality (cf. \cite[Section 3]{Feichtinger-Grochenig-Li-Wang=2021}). To be specific, the Fourier transform of $f\in\ml{S}'_{\mathrm{G-S}}$, i.e. $\ml{F}(f)=\widehat{f}$, fulfills
 \begin{align*}
 	\langle \ml{F}(f),\phi\rangle=\langle f,\ml{F}(\phi)\rangle  \ \mbox{for any}\  \phi\in\ml{S}_{\mathrm{G-S}}.
 \end{align*}
 Remark that $\ml{S}_{\mathrm{G-S}}\subset \ml{S}$ and $\ml{S}'\subset\ml{S}'_{\mathrm{G-S}}$.
 Then, we are going to introduce the distribution (initial data) space $E_{s}^{\alpha}$ (cf. \cite{Chen-Wang-Wang=2023,Chen-Lu-Wang=2023,Wang=2024}) that is a Banach space.
 \begin{defn}\label{Defn-E-space}
 	For any $\alpha,s\in\mb{R}$, let us define
 	\begin{align*}
 		E_s^{\alpha}:=\left\{f\in\ml{S}_{\mathrm{G-S}}': \langle\xi\rangle^s2^{\alpha|\xi|}\widehat{f}(\xi)\in L^2  \right\}
 	\end{align*}
 	endowed with its norm
 	\begin{align*}
 		\|f\|_{E_s^{\alpha}}:=\|\langle\xi\rangle^s2^{\alpha|\xi|}\widehat{f}(\xi)\|_{L^2}.
 	\end{align*}
 \end{defn}
 \begin{remark}
 	We discuss three different spaces $E^{\alpha}_s$ according to the sign of $\alpha$, which greatly changes the regularity of functions.
 	\begin{itemize}
 		\item The case $\alpha=0$ is the Sobolev space, namely, $E_{s}^0=H^s$ with $s\in\mb{R}$.
 		\item The case $\alpha>0$ is strongly related to the Sobolev-Gevrey space consisting of analytic functions. It was applied in studying the Gevrey regularity for solutions of some evolution equations, e.g. \cite{Foias-Temam=1989,Bourdaud-Reissig-Sickel=2003}.   Particularly, the infinitely smooth function space $E_0^{\alpha}$ with $\alpha>0$ was introduced in \cite{Bjock=1966}.
 		\item When $\alpha<0$, it is a rather rough space of distributions due to the inclusion $H^{s'}\subset E^{\alpha}_s$ for any $s,s'\in\mb{R}$ according to $\|f\|_{E_s^{\alpha}}\lesssim \|f\|_{H^{s'}}$. We notice that $\cup_{s'\in\mb{R}}H^{s'}$ is a subset of $E_s^{\alpha}$ if $\alpha<0$.
 		We call it as the rough space because $E^{\alpha}_s$ with $\alpha<0$ contains rougher distributions than any Sobolev space $H^{s'}$.
 	\end{itemize}
 \end{remark}

 We recall some preliminaries for the frequency-uniform decomposition techniques which were introduced by \cite{Wang-Zhao-Guo=2006} in the nonlinear PDEs. For any $k\in\mb{Z}^n$, the frequency-uniform decomposition operator $\square_k$ is defined via
 \begin{align*}
 	\square_kf:=\ml{F}^{-1}\big(\chi_{k+[0,1)^n}\ml{F}(f)\big)
 \end{align*}
 with the characteristic function $\chi_{\mb{A}}=1$ on $\mb{A}\subset\mb{R}^n$. Namely, the frequencies of $\ml{F}(\square_kf)$ are localized in $k+[0,1)^n$. Employing the Plancherel identity and the orthogonality of $\square_k$ for $k\in\mb{Z}^n$, the equivalent expression of $\|f\|_{E^{\alpha}_s}$ arises
 \begin{align*}
 	\|f\|_{E_s^{\alpha}}\approx\left(\sum\limits_{k\in\mb{Z}^n}\left(\langle k\rangle^s2^{\alpha|k|}\|\square_kf\|_{L^2}\right)^2 \right)^{1/2}.
 \end{align*}
 For convenience, we denote the distribution (solution) space by
 \begin{align}\label{norm-rough-type}
 	\|g\|_{\widetilde{L}^{\gamma}(\mb{R}_+,E^{\alpha,s}_{2,2})}:=\left(\sum\limits_{k\in\mb{Z}^n}\left(\langle k\rangle^s2^{\alpha|k|}\|\square_kg\|_{L^{\gamma}_tL^2_x}\right)^2\right)^{1/2}.
 \end{align}

\subsection{Real-valued JMGT equations}\label{Sub-Section-Real-Value-JMGT}
\hspace{5mm}We recall now some physical as well as historical background related to the JMGT equations \eqref{Eq-Complex-JGMT-W} in the real-valued setting and, simultaneously, in the local situation $\ml{A}=-\Delta$ (however, there are a few results on the other cases, for example, the linearized model with $\ml{A}=\Delta^2$ describes the vertical displacement in viscoelastic plates \cite{Conti-Pata-Pellicer-Quintanilla=2020}).

It is well-known that to describe the propagation of sound in viscous thermally relaxing fluids, the modern research of nonlinear acoustics with second sound phenomena considers an approximated mathematical model of fully compressible Navier-Stokes-Cattaneo (NSC) system in irrotational flows. To be specific, by applying the Lighthill scheme of approximations \cite{Lighthill=1956} for the fully compressible NSC system to retain the terms of first- and second-orders with respect to small perturbations around the constant equilibrium state, the well-studied JMGT equations arise. The JMGT equations are named after the early works of F.K. Moore, W.E. Gibson \cite{MooreGibson1960} in 1960, of P.A. Thompson \cite{Thompson1972} in 1972, and deduced by the classical work of P.M. Jordan \cite{Jordan-2014} in 2014 who utilized the Cattaneo law to eliminate an infinite signal speed paradox from the Fourier law of heat conduction in thermoviscous fluids. The JMGT equations are written via
\begin{align}\label{Eq-Real-JGMT-W}
\tau\partial_t^3\Psi+\partial_t^2\Psi-\Delta\Psi-(\delta+\tau)\Delta\partial_t\Psi=\begin{cases}
	\partial_t[\frac{B}{2A}(\partial_t\Psi)^2+(\nabla\Psi)^2],&\mbox{the Kuznetsov-type},\\
	(1+\frac{B}{2A})\partial_t[(\partial_t\Psi)^2],&\mbox{the Westervelt-type},
\end{cases}
\end{align}
where the thermal relaxation $\tau\in\mb{R}_+$ comes from the Cattaneo law of heat conduction, the diffusivity of sound $\delta\in\mb{R}_+$ consists of bulk/shear viscosities from the compressible Navier-Stokes system, and the real-valued scalar function $\Psi=\Psi(t,x)\in\mb{R}$ is referred to the acoustic velocity potential due to the irrotational condition.  Basing on the shape of acoustic field being ``close'' to that of plane wave \cite{Coulouvrat-1992}, one may encounter the approximation $\nabla\Psi\approx \partial_t\Psi$ from the so-called substitution corollary, as allowed under the weakly nonlinear scheme. So, the nonlinearity of Kuznetsov-type \eqref{Eq-Real-JGMT-W}$_1$ can be reduced to the one of Westervelt-type \eqref{Eq-Real-JGMT-W}$_2$, namely, the local nonlinear effect is neglected. These mathematical models have been extensively used in medical and industrial applications of high-intensity
ultra sound, for example, the medical imaging and therapy, ultrasound cleaning and welding (cf. \cite{Abramov-1999,Dreyer-Krauss-Bauer-Ried-2000,Kaltenbacher-Landes-Hoffelner-Simkovics-2002}).

We are going to address a brief review on the JMGT equations \eqref{Eq-Real-JGMT-W} in the whole space $\mb{R}^n$ (for the bounded domains case, we refer the interested reader to \cite{Kaltenbacher-Lasiecka-Marchand-2011,Marchand-McDevitt-Triggiani-2012,Kaltenbacher-Lasiecka-Pos-2012,Conejero-Lizama-Rodenas-2015,Dell-Pata=2017,Kaltenbacher-Nikolic-2019,B-L-2020,Kaltenbacher-Niko-2021,Niko-Winker=2024} and references given therein).  As preparations, some qualitative properties of solutions for the corresponding linearized Cauchy problem, i.e. the Moore-Gibson-Thompson (MGT) equation with the vanishing right-hand side, are well-established in recent years, including sharp energy decay rates \cite{Pellicer-Said-Houari=2019}, optimal growth/decay $L^2$ estimates \cite{Chen-Ikehata=2021,Chen-Takeda=2023}, $L^p-L^q$ estimates \cite{Chen-Gong=2024,Chen-Ma-Qin=2025}, singular limits \cite{Chen-Ikehata=2021,Chen-Gong=2024}. The global in-time well-posedness results of regular small data Sobolev solutions and Besov solutions, respectively, were proved by \cite{Racke-Said-2020} and \cite{Said-Houari=Besov=2022} in $\mb{R}^3$ via energy methods. Then, the author of \cite{Said-Houari=Large-Norm=2022} removed the smallness assumption for higher-order Sobolev data in $\mb{R}^3$ and deduced energy decays with the additional $L^1$ integrability of initial data, while it is still necessary to assume the small size of lower-order Sobolev data. Carrying Sobolev regular small data with the additional $L^1$ integrability, the authors of \cite{Chen-Takeda=2023} derived large time optimal growth/decay estimates and optimal leading terms of global in-time small data Sobolev solutions in $\mb{R}^n$ with any $n\geqslant 1$ via the refined WKB analysis and the Fourier analysis. 

In conclusion, the Sobolev space (or other related regular space) as well as the smallness assumption on initial data seems necessary to demonstrate global in-time existence of solutions for the JMGT equations \eqref{Eq-Real-JGMT-W} at least in $\mb{R}^n$. 

\subsection{Motivations of this manuscript}\label{Sub-Section-Motivations}
\hspace{5mm}Complex-valued nonlinear evolution equations have caught a lot of attention in recent years, including the Navier-Stokes system \cite{Li-Sinai=2008,Sverak=2017}, the Euler system \cite{Albritton-Ogden=2024}, the heat equation \cite{Chen-Wang-Wang=2023} (related to the viscous Constantin-Lax-Majda equation \cite{Constantin-Lax-Majda=1985} with a Fourier support restriction on the initial data), the Schr\"odinger equation \cite{Chen-Lu-Wang=2023}, the Klein-Gordon equation \cite{Wang=2024}, the damped evolution equations \cite{Chen-Reissig=2025}, and the general nonlinear models \cite{Nakanishi-Wang=2025}. It is not surprising that the results and methodologies for complex-valued models are different from those for their corresponding real-valued models (sometimes they can include the real-valued case if their imaginary parts vanish).

Inspired by these papers and the real-valued JMGT equation of Westervelt-type \eqref{Eq-Real-JGMT-W}$_2$ in nonlinear acoustics, in this paper, we are going to investigate global in-time existence of solutions for the complex-valued nonlocal JMGT equations \eqref{Eq-Complex-JGMT-W} in the whole space $\mb{R}^n$, whose initial data mainly belong to one of the following frameworks:
\begin{table}[http]
	\centering	
	\begin{tabular}{cccc}
		\toprule
		Initial Data   & \multirow{2}{*}{Regularity} & \multirow{2}{*}{Size}  & Fourier Support \\
		$\psi_j$ with $j\in\{0,1,2\}$& & &  Restriction\\
		\midrule
\multirow{2}{*}{Framework I}  & \multirow{2}{*}{$E^{\alpha}_s$ with $\alpha<0$} & \multirow{2}{*}{Arbitrarily Large} & \multirow{2}{*}{Subset of First Octant}\\
&&&\\
\midrule
		\multirow{2}{*}{Framework II}   & $H^{s}\cap L^{m_1}$ for $\mathrm{Re}\,\psi_j$ & \multirow{2}{*}{Sufficiently Small} & \multirow{2}{*}{--------}\\
		&  $H^{s}\cap L^{m_2}$ for $\mathrm{Im}\,\psi_j$ & &  \\
		\bottomrule
	\multicolumn{4}{l}{\emph{$*$ The parameters $m_1\neq m_2$ in general denote different integrabilities of initial data.}}\\
	\multicolumn{4}{l}{\emph{$**$ The regularity $s$ heavily depend on the dimension $n$ and the index $\sigma$ of $\ml{A}=(-\Delta)^{\sigma}$.}}
	\end{tabular}
	\caption{Initial data frameworks for the complex-valued JMGT equations \eqref{Eq-Complex-JGMT-W}}
	\label{Table_3}
\end{table}

\noindent We would like to underline that our studies for the complex-valued JMGT equations \eqref{Eq-Complex-JGMT-W} are not simply generalizations of those for the real-valued JMGT equations \eqref{Eq-Real-JGMT-W}, whose reasons (associated with new observations and results) will be stated via the following two motivations.
\medskip

\noindent\textbf{Motivation I: Rough large data solution.} First of all, in the framework I for initial data, let us state an observation in the homogeneous Sobolev space $\dot{H}^s$ by the scaled unknown $\psi_{\lambda}=\psi_{\lambda}(t,x)$ as follows:
\begin{align*}
\psi_{\lambda}(t,x):=\psi(t,\lambda x)\ \mbox{with}\  \mbox{any}\ \lambda>0.
\end{align*}
Applying the scaling property of fractional Laplacians $\langle\ml{A}\psi,\phi\rangle=\langle\lambda^{-2\sigma}\ml{A}\psi_{\lambda},\phi\rangle$ for $\phi\in\ml{S}_{\mathrm{G-S}}$, we may deduce the $\lambda$-dependent JMGT equations
\begin{align}\label{Eq-Complex-Scaled-JGMT-W}
	\begin{cases}
		\tau\partial_t^3\psi_{\lambda}+\partial_t^2\psi_{\lambda}+\lambda^{-2\sigma}\mathcal{A}\psi_{\lambda}+(\delta+\tau)\lambda^{-2\sigma}\mathcal{A}\partial_t\psi_{\lambda}=(1+\frac{B}{2A})\partial_t[(\partial_t\psi_{\lambda})^2],&x\in\mb{R}^n,\ t\in\mb{R}_+,\\
		\psi_{\lambda}(0,x)=\psi_{0,\lambda}(x),\ \partial_t\psi_{\lambda}(0,x)=\psi_{1,\lambda}(x),\ \partial_t^2\psi_{\lambda}(0,x)=\psi_{2,\lambda}(x),&x\in\mb{R}^n,
	\end{cases}
\end{align}
with the initial data $\psi_{j,\lambda}(x)=\psi_j(\lambda x)$ for $j\in\{0,1,2\}$. The derivative of solution in the sense of energy terms fulfills
\begin{align*}
\|\ml{A}^{\frac{1}{2}}\partial_t\psi_{\lambda}(t,\cdot)\|_{\dot{H}^s}=\lambda^{s+\sigma-\frac{n}{2}}\|\ml{A}^{\frac{1}{2}}\partial_t\psi(t,\cdot)\|_{\dot{H}^s},
\end{align*}
which is invariant in $\dot{H}^{s}$ with $s=s_{\mathrm{crit}}:=\frac{n}{2}-\sigma$. Taking $t=0$, the last equality implies
\begin{align*}
\|\psi_{1,\lambda}\|_{\dot{H}^{s+\sigma}}=\lambda^{s-s_{\mathrm{crit}}}\|\psi_1\|_{\dot{H}^{s+\sigma}}\to0 \ \mbox{as}\ \lambda\to+\infty,
\end{align*}
for $\psi_1\in \dot{H}^{s+\sigma}$ with $s<s_{\mathrm{crit}}$. This phenomenon suggests that the scaled solution $\psi_{\lambda}$ can have very small initial data in a suitable $\dot{H}^s$ framework by choosing $\lambda\gg1$, even if the un-scaled initial data are very large. Analogously, it allows to consider arbitrary large initial data in the JMGT equations \eqref{Eq-Complex-JGMT-W} by using the scaling argument, provided that the previous observation is adapted to the rough data space $E^{\alpha}_s$ with $\alpha<0$. Notice that the well-known identity $\|\,[\partial_t\psi(t,\cdot)]^2\|_{L^m}=\|\partial_t\psi(t,\cdot)\|_{L^{2m}}^2$ does not valid anymore in rough spaces. Therefore, to deal with the nonlinearity in $E^{\alpha}_s$-type spaces, we will apply the frequency-uniform decomposition techniques (cf. \cite{Wang-Zhao-Guo=2006,Chen-Wang-Wang=2023}) and the good algebraic structure of $E^{\alpha}_s$-type spaces when the frequency is localized in the first octant (it is allowed by considering the complex-valued initial data). In turn, the Fourier support restriction for initial data occurs.
\medskip

\noindent\textbf{Motivation II: Different integrabilities on initial data.} Let us turn to the framework II for initial data, whose Fourier support restriction is dropped. We now state an equivalent form of complex-valued JMGT equations \eqref{Eq-Complex-JGMT-W} by setting
\begin{align*}
	\psi=\mathrm{Re}\,\psi+i\mathrm{Im}\,\psi:=u+iv
\end{align*}  with its real part $u=u(t,x)\in\mb{R}$ and its imaginary part $v=v(t,x)\in\mb{R}$. Thanks to the fact that
\begin{align*}
	(\partial_t\psi)^2=(\partial_tu)^2-(\partial_tv)^2+2i\partial_tu\partial_tv,
\end{align*} the complex-valued JMGT equations \eqref{Eq-Complex-JGMT-W} can be reduced to the following real-valued strongly coupled systems: 
\begin{align}\label{Eq-System-JMGT}
	\begin{cases}
		\tau\partial_t^3u+\partial_t^2u+\mathcal{A}u+(\delta+\tau)\mathcal{A}\partial_tu=(1+\frac{B}{2A})\partial_t[(\partial_tu)^2-(\partial_tv)^2],&x\in\mb{R}^n,\ t\in\mb{R}_+,\\
		\tau\partial_t^3v+\partial_t^2v+\mathcal{A}v+(\delta+\tau)\mathcal{A}\partial_tv=(2+\frac{B}{A})\partial_t(\partial_tu\partial_tv),&x\in\mb{R}^n,\ t\in\mb{R}_+,\\
		u(0,x)=u_0(x),\ \partial_tu(0,x)=u_1(x),\ \partial_t^2u(0,x)=u_2(x),&x\in\mb{R}^n,\\
		v(0,x)=v_0(x),\ \partial_tv(0,x)=v_1(x),\ \partial_t^2v(0,x)=v_2(x),&x\in\mb{R}^n.
	\end{cases}
\end{align}
Different from the single equations \eqref{Eq-Real-JGMT-W}$_2$, the coupling structure in the nonlinearities appears in the strongly coupled JMGT systems \eqref{Eq-System-JMGT}, namely,
\begin{align*}
\tau\partial_t^3\ml{U}+\partial_t^2\ml{U}+\mathcal{A}\ml{U}+(\delta+\tau)\mathcal{A}\partial_t\ml{U}=(2+\tfrac{B}{A})\begin{pmatrix}
\partial_tu& -\partial_tv\\
\partial_tv& \partial_tu
\end{pmatrix}\partial_t^2\ml{U}
\end{align*}
for the vector unknown $\ml{U}:=(u,v)^{\mathrm{T}}\in\mb{R}^2$. This nonlinear coupling structure will bring some additional difficulties.
 At this moment, it seems possible to assume different integrabilities on the initial data $u_j$ and $v_j$ for $j\in\{0,1,2\}$ in the global in-time existence part. This proof is based on the WKB analysis associated with the Fourier analysis in the linearized equations, and some fractional interpolations (e.g. the fractional Gagliardo-Nirenberg inequality, the fractional Leibniz rule as well as the fractional Sobolev embedding) to estimate our nonlinear terms.

\subsection{Notations}\label{Sub-Section-Notations}
\hspace{5mm}The generic positive constants $c$ and $C$ may vary from line to line. We write $f\lesssim g$ if there exists a positive constant $C$ such that $f\leqslant Cg$, analogously for $f\gtrsim g$. The sharp relation $f\approx g$ holds, if and only if, $g\lesssim f\lesssim g$. We define $[a]_+:=\max\{a,0\}$ being the nonnegative part of $a$ and $\frac{1}{[a]_+}=+\infty$ if $a\leqslant0$. As usual, $m'=\frac{m}{m-1}$ stands for the H\"older conjugate of $m\in[1,+\infty]$. The convolution with respect to the variable $y$ is denoted by $\ast_{(y)}$.

The first octants are written by
\begin{align*}
	\mb{R}^n_{\mathrm{Oct}}:=\{\xi\in\mb{R}^n: \xi_j\geqslant 0\ \mbox{for all}\  j=1,\dots,n\}\  \mbox{and}\  \mb{Z}^n_{\mathrm{Oct}}:=\mb{R}^n_{\mathrm{Oct}}\cap \mb{Z}^n,
\end{align*}
where $\xi_j$ denotes the $j$-th component of vector $\xi\in\mb{R}^n$. Away from the origin, we define
\begin{align*}
	\mb{R}^n_{\mathrm{Oct},R}:=\{ \xi\in\mb{R}^n_{\mathrm{Oct}}: |\xi|_{\infty}\geqslant R>0\},
\end{align*}
in which $\ell^m$ stands for the sequence Lebesgue space with its norm $|\cdot|_{m}$ for $m\in[1,+\infty]$.

\section{Main results}\setcounter{equation}{0}\label{Section-Main-Results}
\hspace{5mm}Before stating our first result, let us introduce a fixed constant $N_0=N_0(\tau,\delta,\sigma)>0$ that is determined in Proposition \ref{Prop-Pointwise} to guarantee the magnitude of large frequencies for the nonlocal MGT equations. The next result includes the rough large data solution if $\alpha<0$ and the Sobolev small data solution if $\alpha=0$.
\begin{theorem}\label{Thm-rough-large-data}
Let $\alpha\leqslant 0$ and $s\geqslant s_{\mathrm{crit}}=\frac{n}{2}-\sigma$ with $\sigma\in\mb{R}_+$. Let us assume
\begin{align*}
(\psi_0,\psi_1,\psi_2)\in E^{\alpha}_{s+\sigma}\times E^{\alpha}_{s+\sigma}\times E^{\alpha}_s\ \mbox{and}\ \psi_1^2\in E_s^{\alpha}
\end{align*}
such that $\mathrm{supp}\,\widehat{\psi}_j\subset\mb{R}^n_{\mathrm{Oct},N_0}$ for $j\in\{0,1,2\}$. Then, there exists $\alpha_0\leqslant \alpha$ such that there is a uniquely determined solution (in the time-derivative sense)
\begin{align*}
\partial_t\psi\in\ml{C}(\mb{R}_+,E^{\alpha_0}_{s+\sigma})\cap \widetilde{L}^1(\mb{R}_+,E^{\alpha_0,s+\sigma}_{2,2})\cap \widetilde{L}^{\infty}(\mb{R}_+,E^{\alpha_0,s+\sigma}_{2,2})
\end{align*}
for the complex-valued nonlocal JMGT equations of Westervelt-type \eqref{Eq-Complex-JGMT-W} equipping $\ml{A}=(-\Delta)^{\sigma}$, if one of the following conditions for the quantity
\begin{align*}
\|(\psi_{0},\psi_{1},\psi_{2})\|_{E^{\alpha}_{s+\sigma}\times E^{\alpha}_{s+\sigma}\times E^{\alpha}_s}+\|\psi_{1}^2\|_{E^{\alpha}_s}=:\epsilon_{\mathrm{Rough}}>0
\end{align*}
 hold:
\begin{itemize}
	\item $\alpha<0$: with any (even large) size of initial data $\epsilon_{\mathrm{Rough}}$;
	\item $\alpha=0$: with the sufficiently small size of initial data $\epsilon_{\mathrm{Rough}}\ll 1$.
\end{itemize}
\end{theorem}
\begin{remark}
By an interpolation, we claim $\partial_t\psi\in\widetilde{L}^p(\mb{R}_+,E_{2,2}^{\alpha_0,s+\sigma})$ for any $p\in[1,+\infty]$. Concerning other regularities of solution $\psi$, one may use the Leibniz integral rule
\begin{align*}
	\psi(t,x)=\int_0^t\partial_t\psi(\theta,x)\mathrm{d}\theta+\psi_0(x)
\end{align*}
and some regularities statements for $\partial_t\psi$ in Theorem \ref{Thm-rough-large-data}.
\end{remark}
\begin{remark}
	The parameter $\alpha$ in $E^{\alpha}_s$ or $E_{2,2}^{\alpha,s}$ always stands for the radius (for instance, the well-known analytic radius if $\alpha>0$). In Theorem \ref{Thm-rough-large-data} if $\alpha<0$, there is a loss on the radius of rough solution $\partial_t\psi$. Specifically, $\alpha_0:=\lambda\alpha<0$ results from the scaling argument for the large data, where $\lambda>1$ relies on $\alpha,s,\sigma,n,N_0$ and the size of initial data. An example of $\lambda$ is explained in \eqref{Example-lambda}. This radius loss phenomenon will expire (i.e. $\alpha_0=\alpha$ or $\lambda=1$) if the sufficiently small size of initial data is assumed additionally. 
\end{remark}

Turning to the small data existence result without Fourier support restrictions, let us introduce the initial data spaces
\begin{align}\label{Data-Spaces}
	\mb{D}_{s,m_j}:=(H^{s+\sigma}\cap L^{m_j})\times (H^{s+\sigma}\cap L^{m_j})\times (H^{s}\cap L^{m_j}),
\end{align}
where $s\in\mb{R}_+$ and $m_j\in[1,2)$ for $j\in\{1,2\}$. Note that the existence of solution for the equivalent strongly coupled systems \eqref{Eq-System-JMGT} guarantees the existence of solution for the complex-valued JMGT equations \eqref{Eq-Complex-JGMT-W}.
\begin{theorem}\label{Thm-regular-small-data}
Let $m_1,m_2\in[1,2)$ and $s>[\frac{n}{2}-\sigma]_+$ with $\sigma\in\mb{R}_+$ such that
\begin{align*}
\tfrac{2}{\max\{m_1,m_2\}}\geqslant\tfrac{1}{\min\{m_1,m_2\}}+\tfrac{\sigma}{n}.
\end{align*}
Let us assume
\begin{align*}
(u_0,u_1,u_2)\in\mb{D}_{s,m_1},\ (v_0,v_1,v_2)\in\mb{D}_{s,m_2}\ \mbox{and}\ (u_1,v_1)\in H^{s+\sigma}\times H^{s+\sigma}.
\end{align*}
Then, there is a uniquely determined Sobolev solution (in the time-derivative sense)
\begin{align*}
(\partial_tu,\partial_tv)^{\mathrm{T}}\in \left(\ml{C}(\mb{R}_+,H^{s+\sigma})\right)^2
\end{align*}
for the nonlocal strongly coupled JMGT systems \eqref{Eq-System-JMGT} equipping $\ml{A}=(-\Delta)^{\sigma}$, if the initial data are sufficiently small
\begin{align*}
\|(u_0,u_1,u_2)\|_{\mb{D}_{s,m_1}}+\|(v_0,v_1,v_2)\|_{\mb{D}_{s,m_2}}+\|(u_1,v_1)\|^2_{H^{s+\sigma}\times H^{s+\sigma}}=:\epsilon_{\mathrm{Regular}}\ll 1.
\end{align*} Furthermore, the solution satisfies the following sharp decay estimates:
\begin{align*}
\|\partial_tw(t,\cdot)\|_{L^2}&\lesssim (1+t)^{-\frac{n(2-m_j)}{4m_j\sigma}}\epsilon_{\mathrm{Regular}},\\
\|\partial_tw(t,\cdot)\|_{\dot{H}^{s+\sigma}}&\lesssim (1+t)^{-\frac{n(2-m_j)}{4m_j\sigma}-\frac{s+\sigma}{2\sigma}}\epsilon_{\mathrm{Regular}},
\end{align*}
for $(w,j)\in\{(u,1),(v,2)\}$.
\end{theorem}
\begin{remark}
Let $v(t,x)\equiv0$, i.e. the imaginary part of $\psi(t,x)\in\mb{C}$ vanishes, in Theorem \ref{Thm-regular-small-data}. It immediately turns to the real-valued nonlocal JMGT equations
\begin{align*}
\tau\partial_t^3u+\partial_t^2u+\mathcal{A}u+(\delta+\tau)\mathcal{A}\partial_tu=(1+\tfrac{B}{2A})\partial_t[(\partial_tu)^2]\ \mbox{in}\ \mb{R}^n,
\end{align*}
whose sufficient condition for the global in-time existence of solution is $m\leqslant \frac{n}{\sigma}$ with $m\in[1,2)$ and $s>[\frac{n}{2}-\sigma]_+$ carrying the small initial data in $\mb{D}_{s,m}$. Especially, in the lower dimensions $n<2\sigma$, this condition can be simply written by $m\in[1,2)$ and $s\in\mb{R}_+$.
\end{remark}
\begin{remark}
If $m_1=m_2=1$, we just require $s>[\frac{n}{2}-\sigma]_+$ and $n\geqslant \sigma$ to guarantee the global in-time existence in Theorem \ref{Thm-regular-small-data}. Particularly, for the classical situation of nonlinear acoustics, i.e. $\sigma=1$, the sufficient condition turns to $s>0$ if $n\leqslant 2$ and $s>\frac{n}{2}-1$ if $n\geqslant 3$, which exactly coincide with those for the real-valued case \cite[Theorem 3.1]{Chen-Takeda=2023}.
\end{remark}
\begin{remark}
Although the small data global in-time existence result of special case $\sigma=1$ equipping $m_1=m_2=1$ coincides with the one of real-valued case \cite{Chen-Takeda=2023}, there are some new effects in the complex-valued JMGT equations \eqref{Eq-Complex-JGMT-W} as follows:
\begin{itemize}
	\item we may assume different integrabilities $m_1\neq m_2$ (they are not necessary to be $1$) for the real and imaginary parts of initial data;
	\item we may get different decay rates of the real and imaginary parts of solution;
	\item we may consider the general situation $\sigma\in\mb{R}_+$ (that can be used in viscoelastic plates \cite{Conti-Pata-Pellicer-Quintanilla=2020}).
\end{itemize}
Concerning other regularities and asymptotic profiles of solution $(u,v)^{\mathrm{T}}$, we believe that one may follow the lengthy computations in \cite{Chen-Takeda=2023} to demonstrate
\begin{align*}
(u,v)^{\mathrm{T}}\in\left(\ml{C}(\mb{R}_+,H^{s+2\sigma})\cap \ml{C}^1(\mb{R}_+,H^{s+\sigma})\cap \ml{C}^2(\mb{R}_+,H^{s})\right),
\end{align*}
moreover,
\begin{align*}
u(t,x)&\sim \ml{F}^{-1}_{\xi\to x}\left(\tfrac{\sin(|\xi|^{\sigma}t)}{|\xi|^{\sigma}}\mathrm{e}^{-\frac{\delta}{2}|\xi|^{2\sigma}t}\right)\int_{\mb{R}^n}\big[u_1(x)+\tau u_2(x)-\tau(1+\tfrac{B}{2A})\big([u_1(x)]^2-[v_1(x)]^2\big)\big]\mathrm{d}x,\\
v(t,x)&\sim \ml{F}^{-1}_{\xi\to x}\left(\tfrac{\sin(|\xi|^{\sigma}t)}{|\xi|^{\sigma}}\mathrm{e}^{-\frac{\delta}{2}|\xi|^{2\sigma}t}\right)\int_{\mb{R}^n}\big[v_1(x)+\tau v_2(x)-\tau(2+\tfrac{B}{A})[u_1(x)v_1(x)]\big]\mathrm{d}x,
\end{align*}
in the $L^2$ framework, but these verifications are beyond the scope of
this paper.
\end{remark}

\section{Global in-time existence of rough large data solution}\setcounter{equation}{0}\label{Section-Global-Large-Data}
\hspace{5mm}This section is organized as follows. As our preparations for studying nonlinear problems, we in Subsection \ref{Sub-Section-Scaled-Linear} contribute to uniformly in-time estimates in $\widetilde{L}^{\gamma}(\mb{R}_+,E_{2,2}^{\alpha,s+\sigma})$ of scaled solution for the $\lambda$-dependent linear MGT equations via pointwise estimates in the Fourier space and the frequency-uniform decomposition techniques. By recalling some preliminaries of $E^{\alpha}_s$-type rough spaces in Subsection \ref{Sub-Section-Tools-rough}, we are able to estimate the nonlinearity and deal with large initial data in rough spaces. Consequently, in Subsection \ref{Sub-Section-Rough-Small-Data} and Subsection \ref{Sub-Section-Rough-regularity}, respectively, we justify an existence result and a regularity result of global in-time rough small data solution for the scaled JMGT equations \eqref{Eq-Complex-Scaled-JGMT-W}. Lastly, in Subsection \ref{Sub-Section-Rough-large-data} by employing the scaling argument, the smallness assumption of initial data in $E^{\alpha}_s$ if $\alpha<0$ is removed to finish the proof of Theorem \ref{Thm-rough-large-data}.

\subsection{Scaled linear MGT equations}\label{Sub-Section-Scaled-Linear}
\subsubsection{Pointwise estimates in the Fourier space}
\hspace{5mm}Let us consider the corresponding linearized models to the JMGT equations \eqref{Eq-Complex-Scaled-JGMT-W} with the vanishing right-hand side, i.e. the $\lambda$-dependent nonlocal MGT equations
\begin{align}\label{Eq-Complex-Scaled-Linear-GMT-W}
	\begin{cases}
		\tau\partial_t^3\varphi_{\lambda}+\partial_t^2\varphi_{\lambda}+\lambda^{-2\sigma}\mathcal{A}\varphi_{\lambda}+(\delta+\tau)\lambda^{-2\sigma}\mathcal{A}\partial_t\varphi_{\lambda}=0,&x\in\mb{R}^n,\ t\in\mb{R}_+,\\
		\varphi_{\lambda}(0,x)=\varphi_{0,\lambda}(x),\ \partial_t\varphi_{\lambda}(0,x)=\varphi_{1,\lambda}(x),\ \partial_t^2\varphi_{\lambda}(0,x)=\varphi_{2,\lambda}(x),&x\in\mb{R}^n,
	\end{cases}
\end{align}
whose partial Fourier transform with respect to the spatial variable ($x\mapsto\xi$) is given by
\begin{align}\label{Eq-Fourier-MGT-eta}
\begin{cases}
\tau\mathrm{d}_t^3\widehat{\varphi}_{\lambda}+\mathrm{d}_t^2\widehat{\varphi}_{\lambda}+|\eta|^{2\sigma}\widehat{\varphi}_{\lambda}+(\delta+\tau)|\eta|^{2\sigma}\mathrm{d}_t\widehat{\varphi}_{\lambda}=0,&\xi\in\mb{R}^n,\ t\in\mb{R}_+,\\
\widehat{\varphi}_{\lambda}(0,\xi)=\widehat{\varphi}_{0,\lambda}(\xi),\ \mathrm{d}_t\widehat{\varphi}_{\lambda}(0,\xi)=\widehat{\varphi}_{1,\lambda}(\xi),\ \mathrm{d}_t^2\widehat{\varphi}_{\lambda}(0,\xi)=\widehat{\varphi}_{2,\lambda}(\xi),&\xi\in\mb{R}^n,
\end{cases}
\end{align}
carrying the new parameter $\eta:=\lambda^{-1}\xi$ for the sake of convenience. Its characteristic equation for $\mu=\mu(|\xi|)$ is addressed as follows:
\begin{align}\label{Characteristic-roots}
	\tau\mu^3+\mu^2+(\delta+\tau)|\eta|^{2\sigma}\mu+|\eta|^{2\sigma}=0.
\end{align}
Clearly, the solution $\varphi_{\lambda}=\varphi_{\lambda}(t,x)$ for the third-order in-time differential equations \eqref{Eq-Complex-Scaled-Linear-GMT-W} can be expressed via
\begin{align}\label{representation-varphi-lambda}
\varphi_{\lambda}(t,x)=\sum\limits_{j\in\{0,1,2\}}\ml{K}_j(t,|D|;\lambda)\varphi_{j,\lambda}(x),
\end{align}
in which the solution's kernels $\ml{K}_j(t,|D|;\lambda)$ have their symbols $\widehat{\ml{K}}_j(t,|\xi|;\lambda)$ that can be obtained by solving the $|\xi|$-dependent ODEs \eqref{Eq-Fourier-MGT-eta}.

\begin{prop}\label{Prop-Pointwise}
Let $\lambda>0$. There exists $N_0=N_0(\tau,\delta,\sigma)>0$ but independent of $\lambda$ such that for any $|\xi|\geqslant N_0\lambda$ the following three statements hold.
\begin{enumerate}[(a)]
	\item The characteristic equation \eqref{Characteristic-roots} has one real root $\mu_1$ and two non-real conjugate roots $\mu_{2,3}$.
	\item The characteristic roots are
	\begin{align*}
	\mu_1=-\tfrac{1}{\delta+\tau}+O(\lambda^{\sigma}|\xi|^{-\sigma}),\ \mu_{2,3}=\pm i\sqrt{\tfrac{\delta+\tau}{\tau}}\lambda^{-\sigma}|\xi|^{\sigma}-\tfrac{\delta}{2\tau(\delta+\tau)}+O(\lambda^{\sigma}|\xi|^{-\sigma}).
	\end{align*}
	\item The solution's kernels in the Fourier space satisfy the next pointwise estimates:
	\begin{align*}
	|\partial_t\widehat{\ml{K}}_0(t,|\xi|;\lambda)|+|\partial_t\widehat{\ml{K}}_1(t,|\xi|;\lambda)|+|\partial_t^2\widehat{\ml{K}}_2(t,|\xi|;\lambda)|&\lesssim\mathrm{e}^{-ct}, \\ 
	|\partial_t\widehat{\ml{K}}_2(t,|\xi|;\lambda)|&\lesssim \lambda^{\sigma}|\xi|^{-\sigma}\mathrm{e}^{-ct}
	\end{align*}
	for $t\in\mb{R}_+$, where the above unexpressed multiplicative constants are independent of $\lambda$.
\end{enumerate}
\end{prop}
\begin{remark}
Taking $\lambda=1$ as well as $\sigma=1$, these results exactly coincide with those in \cite[Proposition 2.1]{Chen-Ikehata=2021} and \cite[Subsection 2.1]{Chen-Takeda=2023}.
\end{remark}
\begin{proof}
Recall that $|\eta|=\lambda^{-1}|\xi|$. The statement (a) is valid due to the negative discriminant of \eqref{Characteristic-roots} satisfying
\begin{align*}
	\triangle_{\mathrm{Disc}}(|\eta|)=-4\tau(\delta+\tau)^3|\eta|^{6\sigma}+O(|\eta|^{4\sigma})<0
\end{align*}
as $|\eta|\geqslant N_0$, and in turn, $|\xi|\geqslant N_0\lambda$.
The statement (b) is derived by using asymptotic expansions straightforwardly. Then, by applying the effective representation of solution in \cite[Equation (23)]{Chen-Takeda=2023}, the statement (c) is proved.  
\end{proof}

\subsubsection{Uniform in-time estimates in the physical space}
\begin{prop}\label{Prop-Linear-Scaled}
Let $\alpha\leqslant0$, $s\in\mb{R}$ and $1\leqslant\gamma\leqslant+\infty$. Suppose that
\begin{align*}
(\varphi_{0,\lambda},\varphi_{1,\lambda},\varphi_{2,\lambda})\in E^{\alpha}_{s+\sigma}\times E^{\alpha}_{s+\sigma}\times E^{\alpha}_s,
\end{align*}
and $\mathrm{supp}\,\widehat{\varphi}_{j,\lambda}\subset\mb{R}^n_{\mathrm{Oct},N_0\lambda}$ for $\lambda>0$ and $j\in\{0,1,2\}$. Then, the solution (in the time-derivative sense) for the $\lambda$-dependent MGT equations \eqref{Eq-Complex-Scaled-Linear-GMT-W} satisfies
\begin{align*}
\|\partial_t\varphi_{\lambda}\|_{\widetilde{L}^{\gamma}(\mb{R}_+,E^{\alpha,s+\sigma}_{2,2})}\lesssim \|\varphi_{0,\lambda}\|_{E^{\alpha}_{s+\sigma}}+\|\varphi_{1,\lambda}\|_{E^{\alpha}_{s+\sigma}}+\lambda^{\sigma}\|\varphi_{2,\lambda}\|_{E^{\alpha}_s},
\end{align*}
where the above unexpressed multiplicative constant is independent of $\lambda$.
\end{prop}
\begin{proof}
Applying the Plancherel identity as well as the pointwise estimates in the statement (c) of Proposition \ref{Prop-Pointwise} associated with the support condition of scaled initial data, one may arrive at
\begin{align*}
\|\square_k\partial_t\ml{K}_0(t,|D|;\lambda)\varphi_{0,\lambda}(\cdot)\|_{L^2}&\lesssim\mathrm{e}^{-ct}\|\square_k\varphi_{0,\lambda}\|_{L^2},\\
\|\square_k\partial_t\ml{K}_1(t,|D|;\lambda)\varphi_{1,\lambda}(\cdot)\|_{L^2}&\lesssim\mathrm{e}^{-ct}\|\square_k\varphi_{1,\lambda}\|_{L^2},\\
\|\square_k\partial_t\ml{K}_2(t,|D|;\lambda)\varphi_{2,\lambda}(\cdot)\|_{L^2}&\lesssim\lambda^{\sigma}\langle k\rangle^{-\sigma}\mathrm{e}^{-ct}\|\square_k\varphi_{2,\lambda}\|_{L^2},
\end{align*}
for all $k\in\mb{Z}^n\cap\mb{R}^n_{\mathrm{Oct},N_0\lambda}$, where we used $\chi_{k+[0,1)^n}|\xi|^{-\sigma}\lesssim \langle k\rangle^{-\sigma}$. Thanks to the integrability $\mathrm{e}^{-ct}\in L^{\gamma}_t$, it immediately follows
\begin{align*}
\sum\limits_{j\in\{0,1,2\}}\|\square_k\partial_t\ml{K}_j(t,|D|;\lambda)\varphi_{j,\lambda}\|_{L^{\gamma}_tL^2_x}\lesssim\|\square_k\varphi_{0,\lambda}\|_{L^2}+\|\square_k\varphi_{1,\lambda}\|_{L^2}+\lambda^{\sigma}\langle k\rangle^{-\sigma}\|\square_k\varphi_{2,\lambda}\|_{L^2}
\end{align*}
for all $k\in\mb{Z}^n\cap\mb{R}^n_{\mathrm{Oct},N_0\lambda}$. Next, multiplying the previous inequality by $\langle k\rangle^{s+\sigma}2^{\alpha|k|}$ and taking the $\ell^2$ norm on $\mb{Z}^n$ (but it will be automatically restricted to $k\in\mb{Z}^n\cap\mb{R}^n_{\mathrm{Oct},N_0\lambda}$), one obtains
\begin{align*}
&\|\partial_t\varphi_{\lambda}\|_{\widetilde{L}^{\gamma}(\mb{R}_+,E_{2,2}^{\alpha,s+\sigma})}\\
&\lesssim\sum\limits_{j\in\{0,1\}}\left(\sum\limits_{k\in\mb{Z}^n}\left(\langle k\rangle^{s+\sigma}2^{\alpha|k|}\|\square_k\varphi_{j,\lambda}\|_{L^2}\right)^2\right)^{1/2}+\lambda^{\sigma}\left(\sum\limits_{k\in\mb{Z}^n}\left(\langle k\rangle^s2^{\alpha|k|}\|\square_k\varphi_{2,\lambda}\|_{L^2}\right)^2\right)^{1/2}\\
&\lesssim\|\varphi_{0,\lambda}\|_{E^{\alpha}_{s+\sigma}}+\|\varphi_{1,\lambda}\|_{E^{\alpha}_{s+\sigma}}+\lambda^{\sigma}\|\varphi_{2,\lambda}\|_{E^{\alpha}_s},
\end{align*}
in which the support condition in $\mb{R}^n_{\mathrm{Oct},N_0\lambda}$ was considered again. The proof is complete.
\end{proof}

\begin{prop}\label{Prop-Inhomogeneous-Scaled}
	Let $\alpha\leqslant0$, $s\in\mb{R}$ and $1\leqslant\gamma\leqslant+\infty$. Suppose that $\mathrm{supp}\,\widehat{g}(t,\cdot)\subset\mb{R}^n_{\mathrm{Oct},N_0\lambda}$ for $\lambda>0$. Then, the inhomogeneous part (in the second-order time-derivative sense) for the $\lambda$-dependent MGT equations \eqref{Eq-Complex-Scaled-Linear-GMT-W} satisfies
	\begin{align*}
		\left\|\int_0^t\partial_t^2\ml{K}_2(t-\theta,|D|;\lambda)g(\theta,x)\mathrm{d}\theta\right\|_{\widetilde{L}^{\gamma}(\mb{R}_+,E_{2,2}^{\alpha,s+\sigma})}\lesssim\|g\|_{\widetilde{L}^1(\mb{R}_+,E_{2,2}^{\alpha,s+\sigma})},
	\end{align*}
	where the above unexpressed multiplicative constant is independent of $\lambda$.
\end{prop}

\begin{proof}
Let us now apply the Plancherel identity again associated with the statement (c) in Proposition \ref{Prop-Pointwise} to deduce
\begin{align*}
\left\|\square_k\int_0^t\partial_t^2\ml{K}_2(t-\theta,|D|;\lambda)g(\theta,x)\mathrm{d}\theta\right\|_{L^{\gamma}_tL^2_x}&\lesssim\left\|\int_{\mb{R}_+}|\partial_t^2\widehat{\ml{K}}_2(t-\theta,|k|;\lambda)|\,\|\square_kg(\theta,\cdot)\|_{L^2_x}\mathrm{d}\theta\right\|_{L^{\gamma}_t}\\
&\lesssim\big\|\mathrm{e}^{-ct}\ast_{(t)}\|\square_kg(t,\cdot)\|_{L^2_x}\big\|_{L^{\gamma}_t}\\
&\lesssim \|\square_kg\|_{L^{1}_tL^2_x}
\end{align*}
for all $k\in\mb{Z}^n\cap\mb{R}^n_{\mathrm{Oct},N_0\lambda}$ thanks to the support condition of $\widehat{g}(t,\cdot)$, where we employed the Young convolution inequality with respect to $t$ and $\mathrm{e}^{-ct}\in L^{\gamma}_t$. Analogously to the proof of Proposition \ref{Prop-Linear-Scaled}, via multiplying the last estimate by $\langle k\rangle^{s+\sigma}2^{\alpha|k|}$ and taking the $\ell^2$ norm on $\mb{Z}^n$, we are able to conclude our desired estimate.
\end{proof}

\subsection{Preliminaries in the $E^{\alpha}_s$-type spaces}\label{Sub-Section-Tools-rough}
\hspace{5mm}Let us begin with recalling a multiplicative estimate of two generalized functions from \cite[Proposition 4.1 with $p=2$, $\beta_p=0$ and $s\mapsto s+\sigma$]{Chen-Reissig=2025}, whose philosophy is originally stated in \cite[Lemma 2.5]{Chen-Wang-Wang=2023}.
\begin{prop}\label{Prop=Algebra-Property}
Let $\alpha\leqslant 0$ and $s\geqslant\frac{n}{2}-\sigma$. Suppose that $\mathrm{supp}\,\widehat{g}_j(t,\cdot)\subset\mb{R}^n_{\mathrm{Oct}}$ for $j\in\{1,2\}$. Then, the following estimate holds:
\begin{align*}
\|g_1g_2\|_{\widetilde{L}^1(\mb{R}_+,E_{2,2}^{\alpha,s+\sigma})}\lesssim\|g_1\|_{\widetilde{L}^{2}(\mb{R}_+,E^{\alpha,s+\sigma}_{2,2})}\|g_2\|_{\widetilde{L}^{2}(\mb{R}_+,E^{\alpha,s+\sigma}_{2,2})}.
\end{align*}
\end{prop}

We next think of \cite[Lemma 2.7]{Chen-Wang-Wang=2023} for a scaling property of $\phi_{\lambda}(x):=\phi(\lambda x)$ in $E^{\alpha}_s$ with $\alpha\leqslant 0$, which implies that $\phi_{\lambda}$ in the  $E^{\alpha}_s$ norm possibly ($s\in\mb{R}$ if $\alpha<0$; $s<\frac{n}{2}$ if $\alpha=0$) vanishes as $\lambda\to+\infty$ provided that the support of $\widehat{\phi}$ is away from the origin.
\begin{prop}\label{Prop=Scaling-Property-1}Let $\alpha\leqslant 0$ and $s\in\mb{R}$. Suppose that $\phi\in E^{\alpha}_s$ and $\mathrm{supp}\,\widehat{\phi}\subset\{\xi\in\mb{R}^n: |\xi|\geqslant N_0\}$ for a constant $N_0>0$. Then, the following estimate holds:
	\begin{align*}
		\|\phi_{\lambda}\|_{E^{\alpha}_s}\lesssim \lambda^{-\frac{n}{2}+\max\{s,0\}}2^{{\alpha(\lambda-1)N_0}}\|\phi\|_{E^{\alpha}_s}
	\end{align*}
	for any $\lambda>1$.
\end{prop}

To end this subsection, we consider \cite[Proposition 4.3 with $\kappa=0$ and $\beta_1=\beta_{\infty}=\sigma$]{Chen-Reissig=2025}, which is a trivial generalization of \cite[Lemma 2.8]{Chen-Wang-Wang=2023}. This proposition indeed addresses that the scaled solution $\psi_{\lambda}$ (i.e. $g$) may control the solution $\psi$ (i.e. $g_{1/\lambda}$) in suitable $E^{\alpha}_s$-type rough spaces with a loss of radius from $\alpha$ to $\lambda\alpha$ for a fixed parameter $\lambda$.
\begin{prop}\label{Prop=Scaling-Property-2}
	Let $\alpha\leqslant 0$ and $s\in\mb{R}$. Suppose that $g\in\widetilde{L}^1(\mb{R}_+,E^{\alpha,s+\sigma}_{2,2})\cap \widetilde{L}^{\infty}(\mb{R}_+,E^{\alpha,s+\sigma}_{2,2})$ with $g_{1/\lambda}(t,x):=g(t,\lambda^{-1}x)$. Then, the following estimate holds:
	\begin{align*}
		\|g_{1/\lambda}\|_{\widetilde{L}^1(\mb{R}_+,E^{\lambda\alpha,s+\sigma}_{2,2})\cap\widetilde{L}^{\infty}(\mb{R}_+,E^{\lambda\alpha,s+\sigma}_{2,2})}\leqslant 2^{(-\alpha)c\lambda}\|g\|_{\widetilde{L}^1(\mb{R}_+,E^{\alpha,s+\sigma}_{2,2})\cap \widetilde{L}^{\infty}(\mb{R}_+,E^{\alpha,s+\sigma}_{2,2})}
	\end{align*}
	for any $\lambda>1$.
\end{prop}

\subsection{Existence of global in-time rough small data scaled solution}\label{Sub-Section-Rough-Small-Data}
\hspace{5mm}We in this subsection consider the scaled JMGT equations \eqref{Eq-Complex-Scaled-JGMT-W} with the rough small initial data having their Fourier support restrictions (with the radius $N_0\lambda$ for $\lambda\geqslant1$), precisely,
\begin{align}\label{Additional-Assumption-Support}
\mathrm{supp}\,\widehat{\psi}_{j,\lambda}\subset\mb{R}^n_{\mathrm{Oct},N_0\lambda}\  \mbox{for}\  j\in\{0,1,2\},
\end{align}
strongly motivated by Proposition \ref{Prop-Linear-Scaled}, whose small size are denoted by
\begin{align*}
\|(\psi_{0,\lambda},\psi_{1,\lambda},\psi_{2,\lambda})\|_{E^{\alpha}_{s+\sigma}\times E^{\alpha}_{s+\sigma}\times E^{\alpha}_s}+\|\psi_{1,\lambda}^2\|_{E^{\alpha}_s}=\epsilon.
\end{align*}
After fixing $\lambda\geqslant1$, we may choose $\epsilon>0$ being sufficiently small.

Let us introduce the global in-time rough solution space (in the time-derivative sense) 
\begin{align*}
\ml{X}_{\alpha,s,\sigma}:=\widetilde{L}^2(\mb{R}_+,E^{\alpha,s+\sigma}_{2,2})\ \mbox{with}\ s\geqslant\tfrac{n}{2}-\sigma 
\end{align*}
carrying its corresponding norm \eqref{norm-rough-type}. We are going to demonstrate the existence of solution in the $\lambda$-weighted set as follows:
\begin{align*}
\ml{B}_{\nu}^{\lambda}:=\left\{\partial_t\psi_{\lambda}\in\ml{X}_{\alpha,s,\sigma}: \mathrm{supp}\,\partial_t\widehat{\psi}_{\lambda}(t,\cdot)\subset\mb{R}^n_{\mathrm{Oct},N_0\lambda}\  \mbox{and}\  \|\partial_t\psi_{\lambda}\|_{\ml{B}_{\nu}^{\lambda}}:=\lambda^{-\sigma}\|\partial_t\psi_{\lambda}\|_{\ml{X}_{\alpha,s,\sigma}}\leqslant\nu 
\right\}
\end{align*}
with a parameter $\nu>0$ that will be determined later (like $\nu\approx \epsilon$).

Then, strongly motivated by the local in-time mild solution (i.e. an equivalent integral form, cf. \cite[Subsection 3.2]{Chen-Takeda=2023}), one may define the following nonlinear integral operator:
\begin{align*}
\ml{N}_{\lambda}: \partial_t\psi_{\lambda}\in\ml{B}_{\nu}^{\lambda}\to\ml{N}_{\lambda}[\partial_t\psi_{\lambda}]:=\partial_t\psi_{\lambda}^{\lin}+\partial_t\psi_{\lambda}^{\nlin},
\end{align*}
where $\partial_t\psi_{\lambda}^{\lin}:=\partial_t\varphi_{\lambda}$ is the solution for its corresponding $\lambda$-dependent MGT equations \eqref{Eq-Complex-Scaled-Linear-GMT-W}, and $\partial_t\psi_{\lambda}^{\nlin}=\partial_t\psi_{\lambda}^{\nlin}(t,x)$ is defined via
\begin{align*}
\partial_t\psi_{\lambda}^{\nlin}(t,x)&:=(1+\tfrac{B}{2A})\int_0^t\partial_t\ml{K}_2(t-\theta,|D|;\lambda)\partial_t[(\partial_t\psi_{\lambda}(\theta,x))^2]\mathrm{d}\theta\\
&\ =-(1+\tfrac{B}{2A})\partial_t\ml{K}_2(t,|D|;\lambda)[\psi_{1,\lambda}(x)]^2-(1+\tfrac{B}{2A})\int_0^t\partial_t^2\ml{K}_2(t-\theta,|D|;\lambda)[\partial_t\psi_{\lambda}(\theta,x)]^2\mathrm{d}\theta
\end{align*}
by employing an integration by parts in regard to $t$ as well as $\ml{K}_2(0,|D|;\lambda)\equiv0\equiv\partial_t\ml{K}_2(0,|D|;\lambda)$.

The support condition \eqref{Additional-Assumption-Support} associated with the regularities of initial data allows to use Proposition \ref{Prop-Linear-Scaled} and obtain
\begin{align*}
\|\partial_t\psi_{\lambda}^{\lin}\|_{\ml{B}_{\nu}^{\lambda}}\lesssim\underbrace{\max\{\lambda^{-\sigma},1\}}_{\leqslant 1}\|(\psi_{0,\lambda},\psi_{1,\lambda},\psi_{2,\lambda})\|_{E^{\alpha}_{s+\sigma}\times E^{\alpha}_{s+\sigma}\times E^{\alpha}_s}.
\end{align*}
The support condition of $\partial_t\widehat{\psi}_{\lambda}^{\lin}(t,\cdot)$ is restricted by those of initial data according to its representation \eqref{representation-varphi-lambda} so that $\partial_t\psi_{\lambda}^{\lin}\in\ml{B}_{\nu}^{\lambda}$. Thanks to the support condition from $\partial_t\psi_{\lambda}\in\ml{B}_{\nu}^{\lambda}$ and $s\geqslant\frac{n}{2}-\sigma$, we are able to apply Proposition \ref{Prop-Inhomogeneous-Scaled} and Proposition \ref{Prop=Algebra-Property} to arrive at
\begin{align*}
\|\partial_t\psi_{\lambda}^{\nlin}\|_{\ml{B}_{\nu}^{\lambda}}&\lesssim\lambda^{-\sigma}\|\partial_t\ml{K}_2(t,|D|;\lambda)\psi_{1,\lambda}^2\|_{\widetilde{L}^2(\mb{R}_+,E_{2,2}^{\alpha,s+\sigma})}\\
&\quad+\lambda^{-\sigma}\left\|\int_0^t\partial_t^2\ml{K}_2(t-\theta,|D|;\lambda)[\partial_t\psi_{\lambda}(\theta,x)]^2\mathrm{d}\theta\right\|_{\widetilde{L}^2(\mb{R}_+,E_{2,2}^{\alpha,s+\sigma})}\\
&\lesssim\|\psi_{1,\lambda}^2\|_{E^{\alpha}_s}+\lambda^{-\sigma}\|\partial_t\psi_{\lambda}\|^2_{\widetilde{L}^2(\mb{R}_+,E_{2,2}^{\alpha,s+\sigma})}.
\end{align*}
Note that the support of $\partial_t\widehat{\psi}_{\lambda}^{\nlin}(t,\cdot)$ is obviously localized in $\mb{R}^n_{\mathrm{Oct},N_0\lambda}$ with the aid of $\mathrm{supp}\,\widehat{g_1g_2}(t,\cdot)\subset\mathrm{supp}\,\widehat{g}_1(t,\cdot)+\mathrm{supp}\,\widehat{g}_2(t,\cdot)$.
That is to say, there exist positive constants $C_0,C_1$ independent of $\lambda\geqslant 1$ such that
\begin{align*}
\|\ml{N}_{\lambda}[\partial_t\psi_{\lambda}]\|_{\ml{B}^{\lambda}_{\nu}}&\lesssim\|(\psi_{0,\lambda},\psi_{1,\lambda},\psi_{2,\lambda})\|_{E^{\alpha}_{s+\sigma}\times E^{\alpha}_{s+\sigma}\times E^{\alpha}_s}+\|\psi_{1,\lambda}^2\|_{E^{\alpha}_s}+\lambda^{-\sigma}\|\partial_t\psi_{\lambda}\|_{\ml{X}_{\alpha,s,\sigma}}^2\\
&\leqslant C_0\epsilon+C_1\lambda^{\sigma}\nu^2.
\end{align*}
Letting $\nu:=2C_0\epsilon$ such that $\nu\leqslant (2C_1\lambda^{\sigma})^{-1}$, we claim
\begin{align}\label{Crucial-01}
	\|\ml{N}_{\lambda}[\partial_t\psi_{\lambda}]\|_{\ml{B}^{\lambda}_{\nu}}\leqslant \nu.
\end{align} 
We have justified $\ml{N}_{\lambda}[\partial_t\psi_{\lambda}]\in\ml{B}_{\nu}^{\lambda}$. By the similar way as the above, one can derive the following Lipschitz condition:
\begin{align}\label{Crucial-02}
&\|\ml{N}_{\lambda}[\partial_t\psi_{\lambda}]-\ml{N}_{\lambda}[\partial_t\widetilde{\psi}_{\lambda}]\|_{\ml{B}^{\lambda}_{\nu}}\notag\\
&\lesssim\lambda^{-\sigma}\left\|\int_0^t\partial_t^2\ml{K}_2(t-\theta,|D|;\lambda)\big([\partial_t\psi_{\lambda}(\theta,x)]^2-[\partial_t\widetilde{\psi}_{\lambda}(\theta,x)]^2\big)\mathrm{d}\theta\right\|_{\widetilde{L}^2(\mb{R}_+,E_{2,2}^{\alpha,s+\sigma})}\notag\\
&\lesssim\lambda^{-\sigma}\left(\|\partial_t\psi_{\lambda}(\partial_t\psi_{\lambda}-\partial_t\widetilde{\psi}_{\lambda})\|_{\widetilde{L}^1(\mb{R}_+,E_{2,2}^{\alpha,s+\sigma})}+\|\partial_t\widetilde{\psi}_{\lambda}(\partial_t\psi_{\lambda}-\partial_t\widetilde{\psi}_{\lambda})\|_{\widetilde{L}^1(\mb{R}_+,E_{2,2}^{\alpha,s+\sigma})}\right)\notag\\
&\lesssim\lambda^{-\sigma}\left(\|\partial_t\psi_{\lambda}\|_{\ml{X}_{\alpha,s,\sigma}}+\|\partial_t\widetilde{\psi}_{\lambda}\|_{\ml{X}_{\alpha,s,\sigma}}\right)\|\partial_t\psi_{\lambda}-\partial_t\widetilde{\psi}_{\lambda}\|_{\ml{X}_{\alpha,s,\sigma}}\notag\\
&\leqslant C_1\nu\lambda^{\sigma}\|\partial_t\psi_{\lambda}-\partial_t\widetilde{\psi}_{\lambda}\|_{\ml{B}^{\lambda}_{\nu}}\notag\\
&\leqslant\tfrac{1}{2}\|\partial_t\psi_{\lambda}-\partial_t\widetilde{\psi}_{\lambda}\|_{\ml{B}^{\lambda}_{\nu}}
\end{align}
for any $\partial_t\psi_{\lambda},\partial_t\widetilde{\psi}_{\lambda}\in\ml{B}_{\nu}^{\lambda}$ carrying the identical initial data.

Choosing the total size of initial data to be small such that
\begin{align}\label{Smallness}
0<\epsilon \leqslant (4C_0C_1\lambda^{\sigma})^{-1} \ \Leftrightarrow \ 0<\nu\leqslant (2C_1\lambda^{\sigma})^{-1}
\end{align} 
 for an arbitrary chosen constant $\lambda\geqslant 1$, we consequently claim that $\ml{N}_{\lambda}: \ml{B}_{\nu}^{\lambda}\to\ml{B}_{\nu}^{\lambda}$ is a contraction mapping since \eqref{Crucial-01} as well as \eqref{Crucial-02} if the support condition \eqref{Additional-Assumption-Support} and the smallness condition \eqref{Smallness} for the initial data hold, in which we assume $\alpha\leqslant 0$ and $s\geqslant \frac{n}{2}-\sigma$. In other words, there is a unique fixed point $\partial_t\psi_{\lambda}\in\ml{B}_{\nu}^{\lambda}$ carrying its integral form for the scaled JMGT equations \eqref{Eq-Complex-Scaled-JGMT-W} with rough small data, that is
\begin{align}\label{Global-mild-solution}
\partial_t\psi_{\lambda}(t,x)&=\sum\limits_{j\in\{0,1,2\}}\partial_t\ml{K}_j(t,|D|;\lambda)\psi_{j,\lambda}(x)-(1+\tfrac{B}{2A})\partial_t\ml{K}_2(t,|D|;\lambda)[\psi_{1,\lambda}(x)]^2\notag\\
&\quad-(1+\tfrac{B}{2A})\int_0^t\partial_t^2\ml{K}_2(t-\theta,|D|;\lambda)[\partial_t\psi_{\lambda}(\theta,x)]^2\mathrm{d}\theta
\end{align}
in the set $\ml{B}_{\nu}^{\lambda}$.

\subsection{Regularity of global in-time rough small data scaled solution}\label{Sub-Section-Rough-regularity}
\hspace{5mm}Under the same hypothesis \eqref{Additional-Assumption-Support} and \eqref{Smallness} as the last subsection, via Proposition \ref{Prop-Linear-Scaled} as well as Proposition \ref{Prop-Inhomogeneous-Scaled} with $\gamma\in\{1,+\infty\}$, the global in-time solution \eqref{Global-mild-solution} satisfies the following estimate:
\begin{align}\label{Ineq-regularity-01}
&\|\partial_t\psi_{\lambda}\|_{\widetilde{L}^1(\mb{R}_+,E^{\alpha,s+\sigma}_{2,2})\cap \widetilde{L}^{\infty}(\mb{R}_+,E^{\alpha,s+\sigma}_{2,2})}\notag\\
&\lesssim\|\psi_{0,\lambda}\|_{E^{\alpha}_{s+\sigma}}+\|\psi_{1,\lambda}\|_{E^{\alpha}_{s+\sigma}}+\lambda^{\sigma}\|\psi_{2,\lambda}\|_{E^{\alpha}_s}+\lambda^{\sigma}\|\psi_{1,\lambda}^2\|_{E^{\alpha}_s}+\|(\partial_t\psi_{\lambda})^2\|_{\widetilde{L}^1(\mb{R}_+,E_{2,2}^{\alpha,s+\sigma})}\notag\\
&\leqslant\lambda^{\sigma}(C_0\epsilon+C_1\lambda^{\sigma}\nu^2)\leqslant \lambda^{\sigma}\nu\leqslant (2C_1)^{-1},
\end{align}
which verifies the regularity $\partial_t\psi_{\lambda}\in \widetilde{L}^1(\mb{R}_+,E^{\alpha,s+\sigma}_{2,2})\cap \widetilde{L}^{\infty}(\mb{R}_+,E^{\alpha,s+\sigma}_{2,2})$.

Due to the fact that $\mathrm{e}^{-ct}\lesssim 1$ in the statement (c) of Proposition \ref{Prop-Pointwise}, we may estimate
\begin{align*}
\langle\xi\rangle^{s+\sigma}2^{\alpha|\xi|}|\partial_t\widehat{\psi}_{\lambda}^{\lin}(t,\xi)|&\lesssim 2^{\alpha|\xi|}\left(\langle\xi\rangle^{s+\sigma}|\widehat{\psi}_{0,\lambda}(\xi)|+\langle\xi\rangle^{s+\sigma}|\widehat{\psi}_{1,\lambda}(\xi)|+\lambda^{\sigma}\langle\xi\rangle^s|\widehat{\psi}_{2,\lambda}(\xi)|\right),\\
\|\partial_t\psi_{\lambda}^{\lin}(t,\cdot)\|_{E^{\alpha}_{s+\sigma}}&\lesssim\|\psi_{0,\lambda}\|_{E^{\alpha}_{s+\sigma}}+\|\psi_{1,\lambda}\|_{E^{\alpha}_{s+\sigma}}+\lambda^{\sigma}\|\psi_{2,\lambda}\|_{E^{\alpha}_s},
\end{align*}
for $\xi\in\mb{R}^n_{\mathrm{Oct},N_0\lambda}$ with $\lambda\geqslant1$ from \eqref{Additional-Assumption-Support}, which yields $\partial_t\psi_{\lambda}^{\lin}\in\ml{C}(\mb{R}_+,E^{\alpha}_{s+\sigma})$. Analogously,
\begin{align*}
\|\partial_t\ml{K}_2(t,|D|;\lambda)[\psi_{1,\lambda}(\cdot)]^2\|_{E^{\alpha}_{s+\sigma}}\lesssim \lambda^{\sigma}\|\psi_{1,\lambda}^2\|_{E^{\alpha}_s}.
\end{align*}
Concerning any $0\leqslant t_0< t_1$, we may get the following inequality:
\begin{align*}
&\left\|\int_{t_0}^t\partial_t^2\ml{K}_2(t-\theta,|D|;\lambda)[\partial_t\psi_{\lambda}(\theta,x)]^2\mathrm{d}\theta\right\|_{L^{\infty}([t_0,t_1],E^{\alpha}_{s+\sigma})}\\
&\lesssim \left\|\int_{t_0}^t\partial_t^2\ml{K}_2(t-\theta,|D|;\lambda)[\partial_t\psi_{\lambda}(\theta,x)]^2\mathrm{d}\theta\right\|_{\widetilde{L}^{\infty}([t_0,t_1],E^{\alpha,s+\sigma}_{2,2})}\\
&\lesssim \|\partial_t\psi_{\lambda}\|_{\widetilde{L}^2([t_0,t_1],E^{\alpha,s+\sigma}_{2,2})}^2,
\end{align*}
where we applied Proposition \ref{Prop=Algebra-Property} in $[t_0,t_1]$. Because of $\partial_t\psi_{\lambda}\in\widetilde{L}^2(\mb{R}_+,E^{\alpha,s+\sigma}_{2,2})$ proved in the last subsection, letting $t_0\to t_1$, it immediately follows that $\partial_t\psi_{\lambda}\in\ml{C}(\mb{R}_+,E^{\alpha}_{s+\sigma})$.

\subsection{Proof of Theorem \ref{Thm-rough-large-data}}\label{Sub-Section-Rough-large-data}
\hspace{5mm}The desired existence and regularity of global in-time small data Sobolev solution, i.e. $\alpha=0$ as well as $\lambda=1$ so that
\begin{align*}
\partial_t\psi\in \ml{C}(\mb{R}_+,H^{s+\sigma})\cap\widetilde{L}^1(\mb{R}_+,H^{s+\sigma})\cap \widetilde{L}^{\infty}(\mb{R}_+,H^{s+\sigma}),
\end{align*}
for the complex-valued JMGT equations \eqref{Eq-Complex-JGMT-W} have been proved in the last two subsections, where we took the sufficiently small size of initial data $0<\epsilon\ll 1$ such that \eqref{Smallness} always holds. For this reason, in the rest of this part, we will focus on the rough large data solution (i.e. the case $\alpha<0$ with $\lambda>1$) to complete the proof of Theorem \ref{Thm-rough-large-data}.

Recalling our assumption $\mathrm{supp}\,\widehat{\psi}_j\subset\mb{R}^n_{\mathrm{Oct},N_0}$ for $j\in\{0,1,2\}$ in Theorem \ref{Thm-rough-large-data}, employing Proposition \ref{Prop=Scaling-Property-1} with $|\xi|\geqslant |\xi|_{\infty}\geqslant N_0$, one may straightforwardly deduce
\begin{align*}
\epsilon&=\|(\psi_{0,\lambda},\psi_{1,\lambda},\psi_{2,\lambda})\|_{E^{\alpha}_{s+\sigma}\times E^{\alpha}_{s+\sigma}\times E^{\alpha}_s}+\|\psi_{1,\lambda}^2\|_{E^{\alpha}_s}\\
&\leqslant C_2\lambda^{-\frac{n}{2}+s+\sigma}2^{\alpha(\lambda-1)N_0}\left(\|(\psi_{0},\psi_{1},\psi_{2})\|_{E^{\alpha}_{s+\sigma}\times E^{\alpha}_{s+\sigma}\times E^{\alpha}_s}+\|\psi_1^2\|_{E^{\alpha}_s}\right)
\end{align*}
with a suitable constant $C_2>0$ independent of $\lambda>1$. Setting $\alpha<0$ as well as $\lambda>1$, there exists a constant 
\begin{align*}
\lambda=\lambda\left(\alpha,s,\sigma,n,N_0,\|(\psi_{0},\psi_{1},\psi_{2})\|_{E^{\alpha}_{s+\sigma}\times E^{\alpha}_{s+\sigma}\times E^{\alpha}_s},\|\psi_{1}^2\|_{E^{\alpha}_s}\right)
\end{align*}
such that
\begin{align*}
\underbrace{\lambda^{-\frac{n}{2}+s+2\sigma}2^{\alpha(\lambda-1)N_0}}_{\begin{subarray}{c}\mbox{an exponential decay}\\ \mbox{with respect to $\lambda\gg1$}\end{subarray}}\leqslant \underbrace{(4C_0C_1C_2)^{-1}\left(\|(\psi_{0},\psi_{1},\psi_{2})\|_{E^{\alpha}_{s+\sigma}\times E^{\alpha}_{s+\sigma}\times E^{\alpha}_s}+\|\psi_{1}^2\|_{E^{\alpha}_s}\right)^{-1}}_{\mbox{a fixed constant independent of $\lambda$}}=:C_{\mathrm{data}}
\end{align*}
for the given data $(\psi_{0},\psi_{1},\psi_{2})\in E^{\alpha}_{s+\sigma}\times E^{\alpha}_{s+\sigma}\times E^{\alpha}_s$ and $\psi_1^2\in E^{\alpha}_s$. For example, when $\alpha\leqslant-\frac{2(s+2\sigma)}{N_0}$, we may take
\begin{align}\label{Example-lambda}
\lambda=\left|\frac{2}{\alpha N_0}\log_2\left(2^{\alpha N_0}C_{\mathrm{data}}\right)\right|+1
\end{align}
to deduce $\lambda\leqslant 2^{\lambda}\leqslant 2^{-\frac{\alpha N_0}{2(s+2\sigma)}\lambda}$ and
\begin{align*}
\lambda^{-\frac{n}{2}+s+2\sigma}2^{\alpha(\lambda-1)N_0}&\leqslant 2^{-\alpha N_0}(\lambda2^{\frac{\alpha N_0}{s+2\sigma}\lambda})^{s+2\sigma}\leqslant 2^{-\alpha N_0+\frac{\alpha N_0}{2}\lambda}\leqslant C_{\mathrm{data}}.
\end{align*}
 It ensures the smallness condition \eqref{Smallness} even for large initial data by taking a suitably large parameter $\lambda$. Moreover, our support condition on the initial data in Theorem \ref{Thm-rough-large-data} shows
\begin{align*}
\mathrm{supp}\,\widehat{\psi}_j\subset\mb{R}^n_{\mathrm{Oct},N_0}\ \Rightarrow\ \mathrm{supp}\,\widehat{\psi}_{j,\lambda}\subset\mb{R}^n_{\mathrm{Oct},N_0\lambda}\ \mbox{for} \ j\in\{0,1,2\},
\end{align*}
via the scaling property of Fourier transform, namely, the condition \eqref{Additional-Assumption-Support} is guaranteed.

Summarizing the above statements, concerning the complex-valued JMGT equations \eqref{Eq-Complex-JGMT-W} even with rough large initial data whose Fourier transforms are supported in $\mb{R}^n_{\mathrm{Oct},N_0}$, according to the ansatz $\partial_t\psi_{\lambda}(t,x)=\partial_t\psi(t,\lambda x)$ by choosing a suitable constant $\lambda>1$, the hypotheses \eqref{Additional-Assumption-Support} and \eqref{Smallness} hold so that the scaled solution $\partial_t\psi_{\lambda}\in\ml{X}_{\alpha,s,\sigma}$ uniquely exists and satisfies \eqref{Global-mild-solution}. That is to say, $\partial_t\psi(t,\lambda x)\in \widetilde{L}^2(\mb{R}_+,E_{2,2}^{\alpha,s+\sigma})$ with $s\geqslant \frac{n}{2}-\sigma$.

Let us eventually consider the regularities of $\partial_t\psi$ by those of $\partial_t\psi_{\lambda}$ obtained in Subsection \ref{Sub-Section-Rough-regularity}. Applying $\partial_t\psi(t,x)\equiv\partial_t\psi_{\lambda}(t,\lambda^{-1}x)$ in Proposition \ref{Prop=Scaling-Property-2} and \eqref{Ineq-regularity-01}, we may arrive at
\begin{align*}
\|\partial_t\psi\|_{\widetilde{L}^1(\mb{R}_+,E^{\lambda\alpha,s+\sigma}_{2,2})\cap \widetilde{L}^{\infty}(\mb{R}_+,E^{\lambda\alpha,s+\sigma}_{2,2})}&\leqslant 2^{(-\alpha)c\lambda} \|\partial_t\psi_{\lambda}\|_{\widetilde{L}^1(\mb{R}_+,E^{\alpha,s+\sigma}_{2,2})\cap \widetilde{L}^{\infty}(\mb{R}_+,E^{\alpha,s+\sigma}_{2,2})}\\
&\leqslant (2C_1)^{-1}2^{(-\alpha)c\lambda}<+\infty.
\end{align*}
For another, via the Plancherel identity and the change of variable $\lambda\xi\mapsto \xi$, one derives
\begin{align*}
\|\partial_t\psi(t,\cdot)\|_{E^{\lambda\alpha}_{s+\sigma}}
&=\lambda^{\frac{n}{2}}\|\langle\lambda^{-1}\xi\rangle^{s+\sigma}2^{\alpha|\xi|}\partial_t\widehat{\psi}_{\lambda}(t,\xi)\|_{L^2}\\
&\leqslant \lambda^{\frac{n}{2}}\|\partial_t\psi_{\lambda}(t,\cdot)\|_{E^{\alpha}_{s+\sigma}}<+\infty.
\end{align*}
Taking $\alpha_0:=\lambda\alpha$ with a fixed $\lambda>1$, we complete the proof of Theorem \ref{Thm-rough-large-data}.

\section{Global in-time existence of regular small data solution}\setcounter{equation}{0}\label{Section-Global-Small-Data}
\hspace{5mm}This section is organized as follows. As our preparations for investigating nonlinear problems, we in Subsection \ref{Sub-Section-Linear} derive polynomially decay estimates in the $L^2$ framework of solution for the linear MGT equations via the WKB analysis and the Fourier analysis. Our philosophy and main tools of proof are stated in Subsection \ref{Subs-Section-Philosphy}. By estimating the initial data in Subsection \ref{Sub-Section-Est-Data} and the nonlinear terms in Subsection \ref{Sub-Section-Est-Nonlinear-Terms} via the fractional Gagliardo-Nirenberg inequality, the fractional Leibniz rule and the fractional Sobolev embedding, we complete the proof of Theorem \ref{Thm-regular-small-data} in a suitable time-weighted Sobolev space lastly in Subsection \ref{Sub-Section-Proof-Regular-Small-Solution}.

\subsection{Linear MGT equations}\label{Sub-Section-Linear}
\hspace{5mm}The strongly coupled systems \eqref{Eq-System-JMGT} are formulated via the nonlinear coupled structure on the right-hand sides to the identical linear parts for $w=w(t,x)\in\mb{R}$ as follows:
\begin{align}\label{Eq-Linearized-MGT}
\begin{cases}
\tau\partial_t^3w+\partial_t^2w+\ml{A}w+(\delta+\tau)\ml{A}\partial_tw=0,&x\in\mb{R}^n,\ t\in\mb{R}_+,\\
w(0,x)=w_0(x),\ \partial_tw(0,x)=w_1(x),\ \partial_t^2w(0,x)=w_2(x),&x\in\mb{R}^n,
\end{cases}
\end{align}
in which we took $w\in\{u,v\}$.
For this reason, we subsequently study sharp polynomially decay estimates of $\partial_tw(t,\cdot)$ in the $\dot{H}^{s+\sigma}$ norm with/without the additional $L^m$ integrability ($m\in[1,2)$ in general) of initial data.

By using the partial Fourier transform with respect to the spatial variable ($x\mapsto\xi$) as usual for the last Cauchy problem, one immediately obtains
\begin{align}\label{Eq-01}
\begin{cases}
\tau\mathrm{d}_t^3\widehat{w}+\mathrm{d}_t^2\widehat{w}+|\xi|^{2\sigma}\widehat{w}+(\delta+\tau)|\xi|^{2\sigma}\mathrm{d}_t\widehat{w}=0,&\xi\in\mb{R}^n,\ t\in\mb{R}_+,\\
\widehat{w}(0,\xi)=\widehat{w}_0(\xi),\ \mathrm{d}_t\widehat{w}(0,\xi)=\widehat{w}_1(\xi),\ \mathrm{d}^2_t\widehat{w}(0,\xi)=\widehat{w}_2(\xi),&\xi\in\mb{R}^n.
\end{cases}
\end{align}
Its corresponding characteristic cubic equation \eqref{Characteristic-roots} equipping $|\eta|=|\xi|$ has the following pairwise distinct roots (taking $\lambda=1$ in Proposition \ref{Prop-Pointwise}):
\begin{align*}
\mu_1=-\tfrac{1}{\delta+\tau}+O(|\xi|^{-\sigma}),\ \mu_{2,3}=\pm i\sqrt{\tfrac{\delta+\tau}{\tau}}|\xi|^{\sigma}-\tfrac{\delta}{2\tau(\delta+\tau)}+O(|\xi|^{-\sigma})\ \mbox{for}\ |\xi|\geqslant N_0,
\end{align*}
and (via asymptotic expansions associated with the discriminant $\triangle_{\mathrm{Disc}}(|\xi|)=-4|\xi|^{2\sigma}+O(|\xi|^{4\sigma})<0$ as $|\xi|\leqslant \varepsilon_0$, analogously to \cite{Chen-Ikehata=2021,Chen-Takeda=2023}) as follows:
\begin{align*}
\mu_1=-\tfrac{1}{\tau}+O(|\xi|^{2\sigma}),\ \mu_{2,3}=\pm i|\xi|^{\sigma}-\tfrac{\delta}{2}|\xi|^{2\sigma}+O(|\xi|^{3\sigma})\ \mbox{for}\ |\xi|\leqslant\varepsilon_0.
\end{align*}
Moreover, they satisfy $\mathrm{Re}\,\mu_j<0$ for $j\in\{1,2,3\}$ as $\varepsilon_0\leqslant|\xi|\leqslant N_0$ due to the continuity of characteristic roots with respect to $|\xi|$. The solution (in the time-derivative sense) for the nonlocal MGT equations \eqref{Eq-Linearized-MGT} is expressed by
\begin{align*}
\partial_tw(t,x)=\sum\limits_{j\in\{0,1,2\}}\partial_t\ml{K}_j(t,|D|)w_j(x),
\end{align*}
in which the solution's kernels $\ml{K}_j(t,|D|)$ have their symbols $\widehat{\ml{K}}_j(t,|\xi|)$ that can be obtained by solving the $|\xi|$-dependent ODEs \eqref{Eq-01} explicitly.

\begin{prop}\label{Prop-L2-Lm-Real}
Let $s\in\{-\sigma\}\cup[0,+\infty)$ and $m\in[1,2)$. Then, the solution (in the time-derivative sense) for the MGT equations \eqref{Eq-Linearized-MGT} satisfies
\begin{align*}
\|\partial_tw(t,\cdot)\|_{\dot{H}^{s+\sigma}}&\lesssim (1+t)^{-\frac{n(2-m)}{4m\sigma}-\frac{s+\sigma}{2\sigma}}\|(w_0,w_1,w_2)\|_{(\dot{H}^{s+\sigma}\cap L^m)\times (\dot{H}^{s+\sigma}\cap L^m)\times (\dot{H}^{\max\{s,0\}}\cap L^m)}.
\end{align*}
Remark that the last estimate also holds for inhomogeneous Sobolev data with the additional $L^m$ integrability.
\end{prop}
\begin{proof}
To begin with the proof, we concentrate on the small frequencies $|\xi|\leqslant \varepsilon_0$, which allows to represent the solution (cf. the time-derivative of \cite[Equation (23)]{Chen-Takeda=2023}) as follows:
\begin{align*}
	\partial_t\widehat{w}&=\frac{-(\mu_{\mathrm{I}}^2+\mu_{\mathrm{R}}^2)\widehat{w}_0+2\mu_{\mathrm{R}}\widehat{w}_1-\widehat{w}_2}{2\mu_{\mathrm{R}}\mu_1-\mu_{\mathrm{I}}^2-\mu_{\mathrm{R}}^2-\mu_1^2}\mu_1\mathrm{e}^{\mu_1t}\\
	&\quad+\frac{(2\mu_{\mathrm{R}}\mu_1-\mu_1^2)\widehat{w}_0-2\mu_{\mathrm{R}}\widehat{w}_1+\widehat{w}_2}{2\mu_{\mathrm{R}}\mu_1-\mu_{\mathrm{I}}^2-\mu_{\mathrm{R}}^2-\mu_1^2}\big(\cos(\mu_{\mathrm{I}}t)\mu_{\mathrm{R}}-\sin(\mu_{\mathrm{I}}t)\mu_{\mathrm{I}}\big)\mathrm{e}^{\mu_{\mathrm{R}}t}\\
	&\quad+\frac{\mu_1(\mu_{\mathrm{R}}\mu_1+\mu_{\mathrm{I}}^2-\mu_{\mathrm{R}}^2)\widehat{w}_0+(\mu_{\mathrm{R}}^2-\mu_{\mathrm{I}}^2-\mu_1^2)\widehat{w}_1-(\mu_{\mathrm{R}}-\mu_1)\widehat{w}_2}{\mu_{\mathrm{I}}(2\mu_{\mathrm{R}}\mu_1-\mu_{\mathrm{I}}^2-\mu_{\mathrm{R}}^2-\mu_1^2)}\big(\sin(\mu_{\mathrm{I}}t)\mu_{\mathrm{R}}+\cos(\mu_{\mathrm{I}}t)\mu_{\mathrm{I}}\big)\mathrm{e}^{\mu_{\mathrm{R}}t},
\end{align*}
where we put $\mu_1=-\frac{1}{\tau}+O(|\xi|^{2\sigma})$, $\mu_{\mathrm{R}}=-\frac{\delta}{2}|\xi|^{2\sigma}+O(|\xi|^{4\sigma})$ and $\mu_{\mathrm{I}}=|\xi|^{\sigma}+O(|\xi|^{3\sigma})$. Then, let us plug these asymptotic expansions into the above expression. By carrying out straightforward but tedious computations, the solution's kernels in the Fourier space satisfy the next pointwise estimates:
\begin{align*}
|\partial_t\widehat{\ml{K}}_0(t,|\xi|)|&\lesssim|\xi|^{\sigma}\mathrm{e}^{-c|\xi|^{2\sigma}t}+|\xi|^{2\sigma}\mathrm{e}^{-ct},\\
|\partial_t\widehat{\ml{K}}_1(t,|\xi|)|&\lesssim\mathrm{e}^{-c|\xi|^{2\sigma}t}+|\xi|^{2\sigma}\mathrm{e}^{-ct},\\
|\partial_t\widehat{\ml{K}}_2(t,|\xi|)|&\lesssim\mathrm{e}^{-c|\xi|^{2\sigma}t}+\mathrm{e}^{-ct},
\end{align*}
for any $t\geqslant 0$ and $|\xi|\leqslant \varepsilon_0$. Next, by using the H\"older inequality with $\frac{1}{2}=\frac{2-m}{2m}+\frac{1}{m'}$ and the Hausdorff-Young inequality from $L^{m'}_{\xi}$ to $L^m_x$ for $m\in[1,2)$, we are able to derive
\begin{align*}
\|\chi_{[0,\varepsilon_0]^n}\partial_tw(t,\cdot)\|_{\dot{H}^{s+\sigma}}&=\|\chi_{[0,\varepsilon_0]^n}|\xi|^{s+\sigma}\partial_t\widehat{w}(t,\xi)\|_{L^2}\\
&\lesssim\|\chi_{[0,\varepsilon_0]^n}|\xi|^{s+\sigma}(\mathrm{e}^{-c|\xi|^{2\sigma}t}+\mathrm{e}^{-ct})\|_{L^{\frac{2m}{2-m}}}\|(\widehat{w}_0,\widehat{w}_1,\widehat{w}_2)\|_{L^{m'}\times L^{m'}\times L^{m'}}\\
&\lesssim\left(\int_0^{\varepsilon_0}r^{\frac{2m(s+\sigma)}{2-m}+n-1}\mathrm{e}^{-\frac{2mc}{2-m}r^{2\sigma}t}\mathrm{d}r\right)^{\frac{2-m}{2m}}\|(w_0,w_1,w_2)\|_{L^m\times L^m\times L^m}\\
&\lesssim (1+t)^{-\frac{n(2-m)}{4m\sigma}-\frac{s+\sigma}{2\sigma}}\|(w_0,w_1,w_2)\|_{L^m\times L^m\times L^m}.
\end{align*}
Note that $s\in\{-\sigma\}\cup[0,+\infty)$, implying $\frac{2m(s+\sigma)}{2-m}+n-1\geqslant 0$, surely guarantees
\begin{align*}
\int_0^{\varepsilon_0}r^{\frac{2m(s+\sigma)}{2-m}+n-1}\mathrm{e}^{-\tilde{c}r^{2\sigma}t}\mathrm{d}r\lesssim\begin{cases}
1&\mbox{if}\ t\leqslant t_0,\\
t^{-\frac{1}{2\sigma}(\frac{2m(s+\sigma)}{2-m}+n)}\int_0^{+\infty}\theta^{\frac{2m(s+\sigma)}{2-m}+n-1}\mathrm{e}^{-\tilde{c}\theta^{2\sigma}}\mathrm{d}\theta\lesssim t^{-\frac{1}{2\sigma}(\frac{2m(s+\sigma)}{2-m}+n)}&\mbox{if}\ t\geqslant t_0, 
\end{cases}
\end{align*}
via the new variable $\theta=rt^{\frac{1}{2\sigma}}$.
Secondly, from the statement (c) of Proposition \ref{Prop-Pointwise} with $\lambda=1$ for $|\xi|\geqslant N_0$, and
\begin{align*}
|\partial_t\widehat{\ml{K}}_0(t,|\xi|)|+|\partial_t\widehat{\ml{K}}_1(t,|\xi|)|+|\partial_t\widehat{\ml{K}}_2(t,|\xi|)|\lesssim\mathrm{e}^{-ct}\ \mbox{for}\ \varepsilon_0\leqslant |\xi|\leqslant N_0,
\end{align*}
we may easily obtain
\begin{align*}
\|(1-\chi_{[0,\varepsilon_0]^n})\partial_tw(t,\cdot)\|_{\dot{H}^{s+\sigma}}\lesssim\mathrm{e}^{-ct}\|(w_0,w_1,w_2)\|_{\dot{H}^{s+\sigma}\times \dot{H}^{s+\sigma}\times \dot{H}^s}.
\end{align*}
Due to the localization $|\xi|>\varepsilon_0$ leading to $|\xi|\approx\langle\xi\rangle$, the previous estimate still holds for inhomogeneous Sobolev data.
The summary of last estimates completes our proof.
\end{proof}

\begin{prop}\label{Prop-L2-inhomogeneous}
Let $s\in\{-\sigma\}\cup[0,+\infty)$ and $m\in[1,2)$. Then, the inhomogeneous part (in the second-order time-derivative sense) for the MGT equations \eqref{Eq-Linearized-MGT} satisfies
\begin{align*}
\|\partial_t^2\ml{K}_2(t,|D|)g_0(\cdot)\|_{\dot{H}^{s+\sigma}}&\lesssim (1+t)^{-\frac{1}{2}}\|g_0\|_{\dot{H}^{s+\sigma}},\\
	\|\partial_t^2\ml{K}_2(t,|D|)g_0(\cdot)\|_{\dot{H}^{s+\sigma}}&\lesssim (1+t)^{-\frac{n(2-m)}{4m\sigma}-\frac{s+2\sigma}{2\sigma}}\|g_0\|_{\dot{H}^{s+\sigma}\cap L^m}.
\end{align*}
\end{prop}
\begin{proof}
Thanks to the pointwise estimate
\begin{align*}
|\partial_t^2\widehat{\ml{K}}_2(t,|\xi|)|\lesssim|\xi|^{\sigma}\mathrm{e}^{-c|\xi|^{2\sigma}t}+\mathrm{e}^{-ct}
\end{align*}
for any $t\geqslant0$ and $|\xi|\leqslant\varepsilon_0$, we may follow the proof of Proposition \ref{Prop-L2-Lm-Real} to finish this one. Precisely,
\begin{align*}
\|\partial_t^2\ml{K}_2(t,|D|)g_0(\cdot)\|_{\dot{H}^{s+\sigma}}&\lesssim\|\partial_t^2\widehat{\ml{K}}_2(t,|\xi|)\|_{L^{\infty}}\|\,|\xi|^{s+\sigma}\widehat{g}_0\|_{L^2}\\
&\lesssim\left((1+t)^{-\frac{1}{2}}+\mathrm{e}^{-ct}\right)\|g_0\|_{\dot{H}^{s+\sigma}},
\end{align*}
moreover,
\begin{align*}
\|\partial_t^2\ml{K}_2(t,|D|)g_0(\cdot)\|_{\dot{H}^{s+\sigma}}&\lesssim \|\chi_{[0,\varepsilon_0]^n}|\xi|^{s+\sigma}(|\xi|^{\sigma}\mathrm{e}^{-c|\xi|^{2\sigma}t}+\mathrm{e}^{-ct})\|_{L^{\frac{2m}{2-m}}}\|g_0\|_{L^m}+\mathrm{e}^{-ct}\|g_0\|_{\dot{H}^{s+\sigma}}\\
&\lesssim (1+t)^{-\frac{n(2-m)}{4m\sigma}-\frac{s+2\sigma}{2\sigma}}\|g_0\|_{L^m}+\mathrm{e}^{-ct}\|g_0\|_{\dot{H}^{s+\sigma}},
\end{align*}
are considered.
\end{proof}

\begin{remark}
Let $\sigma=1$ and $m=1$. One may notice that Proposition \ref{Prop-L2-Lm-Real} coincides with \cite[Estimates (14) and (17) when $\ell=1$]{Chen-Takeda=2023}; Proposition \ref{Prop-L2-inhomogeneous} coincides with \cite[Estimates (18) and (19) when $\ell=1$]{Chen-Takeda=2023}.
\end{remark}

\subsection{Philosophy of our proof for regular small data}\label{Subs-Section-Philosphy}
\hspace{5mm}Let us construct the time-weighted Sobolev space of solution $\ml{V}=\ml{V}(t,x)$ in the vector sense $\ml{V}:=(\partial_tu,\partial_tv)^{\mathrm{T}}$ by
\begin{align*}
\ml{Y}^{s}(T):=\ml{Y}_1^{s}(T)\times\ml{Y}_2^{s}(T)\ \mbox{with}\ \ml{Y}_j^{s}(T):=\ml{C}([0,T],H^{s+\sigma})
\end{align*}
for any time $T>0$ equipped the corresponding norms
\begin{align*}
\|\partial_tw\|_{\ml{Y}_j^{s}(T)}:=\sup\limits_{t\in[0,T]}\left( (1+t)^{\frac{n(2-m_j)}{4m_j\sigma}}\|\partial_tw(t,\cdot)\|_{L^2}+(1+t)^{\frac{n(2-m_j)}{4m_j\sigma}+\frac{s+\sigma}{2\sigma}}\|\partial_tw(t,\cdot)\|_{\dot{H}^{s+\sigma}}\right)
\end{align*}
for $w\in\{u,v\}$ and $j\in\{1,2\}$. The above time-dependent  polynomial growth coefficients are strongly motivated by Proposition \ref{Prop-L2-Lm-Real}.

Similarly to Subsection \ref{Sub-Section-Rough-Small-Data}, one may define the following nonlinear integral operator:
\begin{align*}
\ml{N}:\ml{V}\in\ml{Y}^{s}(T)\to\ml{N}[\ml{V}]:= (\partial_tu^{\lin}+\partial_tu^{\nlin},\partial_tv^{\lin}+\partial_tv^{\nlin})^{\mathrm{T}},
\end{align*}
in which $\partial_tw$ with $w\in\{u^{\lin},v^{\lin}\}$ is the solution for its corresponding linearized MGT equations \eqref{Eq-Linearized-MGT}. Moreover, the nonlinear parts $\partial_tu^{\nlin}=\partial_tu^{\nlin}(t,x)$ and $\partial_tv^{\nlin}=\partial_tv^{\nlin}(t,x)$ are expressed, respectively, via
\begin{align}
\partial_tu^{\nlin}(t,x)&:=-(1+\tfrac{B}{2A})\partial_t\ml{K}_2(t,|D|)\big([u_1(x)]^2-[v_1(x)]^2\big)\notag\\
&\ \quad-(1+\tfrac{B}{2A})\int_0^t\partial_t^2\ml{K}_2(t-\theta,|D|)\big([\partial_tu(\theta,x)]^2-[\partial_tv(\theta,x)]^2\big)\mathrm{d}\theta\label{Nonlinear-Rep-01}
\end{align}
and
\begin{align}
\partial_tv^{\nlin}(t,x)&:=-(2+\tfrac{B}{A})\partial_t\ml{K}_2(t,|D|)[u_1(x)v_1(x)]\notag\\
&\ \quad-(2+\tfrac{B}{A})\int_0^t\partial_t^2\ml{K}_2(t-\theta,|D|)[\partial_tu(\theta,x)\partial_tv(\theta,x)]\mathrm{d}\theta\label{Nonlinear-Rep-02}
\end{align}
with the aid of integrations by parts in regard to $t$. Note that $\ml{K}_2(0,|D|)\equiv0\equiv\partial_t\ml{K}_2(0,|D|)$.

In the forthcoming subsections, we are going to demonstrate global in-time existence and uniqueness results of regular small data Sobolev solution for the strongly coupled JMGT systems \eqref{Eq-System-JMGT} via proving a fixed point of operator $\ml{N}$, namely, $\ml{N}[\ml{V}]\in\ml{Y}^{s}(T)$ for any $T>0$. To achieve this aim, setting
\begin{align*}
\Lambda_{1,2}^{u,v}:=\{(u,1),(v,2)\},
\end{align*}
 we would like to justify the following two uniformly in-time $T$ inequalities:
\begin{align}
\|\ml{N}[\ml{V}]\|_{\ml{Y}^{s}(T)}&\lesssim\sum\limits_{(w,j)\in\Lambda_{1,2}^{u,v}}\|(w_0,w_1,w_2)\|_{\mb{D}_{s,m_j}}+\|(u_1,v_1)\|^2_{H^{s+\sigma}\times H^{s+\sigma}}+\|\ml{V}\|_{\ml{Y}^{s}(T)}^2,\label{Crucial-03}\\
\|\ml{N}[\ml{V}]-\ml{N}[\widetilde{\ml{V}}]\|_{\ml{Y}^{s}(T)}&\lesssim\|\ml{V}-\widetilde{\ml{V}}\|_{\ml{Y}^{s}(T)}\left(\|\ml{V}\|_{\ml{Y}^{s}(T)}+\|\widetilde{\ml{V}}\|_{\ml{Y}^{s}(T)}\right),\label{Crucial-04}
\end{align}
for any $\ml{V},\widetilde{\ml{V}}\in\ml{Y}^{s}(T)$ carrying the identical initial data, where the data spaces $\mb{D}_{s,m_1}$ and $\mb{D}_{s,m_2}$ are defined in \eqref{Data-Spaces}. Taking the sufficiently small size of initial data (in their function spaces) whose philosophy is similar to \eqref{Smallness}, we then combine \eqref{Crucial-03} and \eqref{Crucial-04} to conclude that there exists a global in-time regular small data Sobolev solution $(\partial_tu,\partial_tv)^{\mathrm{T}}\in \ml{Y}^{s}(+\infty)$ according to the standard Banach fixed point argument. As our byproduct, it will also imply some polynomially decay estimates of solution because the definition of $\ml{Y}^{s}(+\infty)$ as well as
\begin{align*}
\|(\partial_tu,\partial_tv)^{\mathrm{T}}\|_{\ml{Y}^{s}(+\infty)}&\lesssim\sum\limits_{(w,j)\in\Lambda_{1,2}^{u,v}}\|(w_0,w_1,w_2)\|_{\mb{D}_{s,m_j}}+\|(u_1,v_1)\|^2_{H^{s+\sigma}\times H^{s+\sigma}}
\end{align*}
uniformly in-time $T$ where the smallness of initial data are considered in the last line.

We recall the fractional Gagliardo-Nirenberg inequality \cite{Hajaiej-Molinet-Ozawa-Wang-2011} and the fractional Leibniz rule \cite{Grafakos-Oh-2014}, respectively, which will be used to estimate the nonlinearities later.
\begin{prop}\label{fractionalgagliardonirenbergineq}
	Let $p,p_0,p_1\in(1,+\infty)$ and $\kappa\in[0,s)$ with $s\in(0,+\infty)$. Then, the following fractional Gagliardo-Nirenberg inequality holds:
	\begin{align*}
		\|f\|_{\dot{H}^{\kappa}_{p}}\lesssim\|f\|_{L^{p_0}}^{1-\beta}\|f\|^{\beta}_{\dot{H}^{s}_{p_1}},
	\end{align*}
	where  $\beta=\frac{\frac{1}{p_0}-\frac{1}{p}+\frac{\kappa}{n}}{\frac{1}{p_0}-\frac{1}{p_1}+\frac{s}{n}}$ and $ \beta\in[\frac{\kappa}{s},1]$.
\end{prop}

\begin{prop}\label{fractionleibnizrule}
	Let $s\in(0,+\infty)$, $r\in[1,+\infty]$ and $p_1,p_2,q_1,q_2\in(1,+\infty]$ satisfying the relation $\frac{1}{r}=\frac{1}{p_1}+\frac{1}{p_2}=\frac{1}{q_1}+\frac{1}{q_2}$. Then, the following fractional Leibniz rule holds:
	\begin{align*}
		\|fg\|_{\dot{H}^{s}_{r}}\lesssim \|f\|_{\dot{H}^{s}_{p_1}}\|g\|_{L^{p_2}}+\|f\|_{L^{q_1}}\|g\|_{\dot{H}^{s}_{q_2}}.
	\end{align*}
\end{prop}

The next embedding result is stated in \cite[Exercise 6.1.2]{Grafakos=2009} without its proof (as an exercise). We next demonstrate it rigorously for the sake of our completeness. Remark that \cite[Lemma A.1]{D'Abbicco-Ebert-Lucente=2017} is a direct consequence of next proposition by using the Young inequality, but now the parameters $\alpha_0,\beta_0$ are more flexible.
\begin{prop}\label{fractionembedd} Let $-\infty<2\alpha_0<n<2\beta_0<+\infty$. Then, the following fractional Sobolev embedding holds:
	\begin{equation*}
		\|f\|_{L^{\infty}}\lesssim\|f\|_{\dot{H}^{\alpha_0}}^{\frac{2\beta_0-n}{2(\beta_0-\alpha_0)}}\|f\|_{\dot{H}^{\beta_0}}^{\frac{n-2\alpha_0}{2(\beta_0-\alpha_0)}}.
	\end{equation*}
\end{prop}
\begin{proof}
	Let us start with separating the next integral into two parts:
	\begin{align*}
		\|\widehat{f}(\lambda_0\xi)\|_{L^1}=\left(\int_{|\xi|\leqslant 1}+\int_{|\xi|> 1}\right)|\widehat{f}(\lambda_0\xi)|\mathrm{d}\xi
	\end{align*}
	with a suitable parameter $\lambda_0>0$ that will be chosen later. By employing the Cauchy-Schwarz inequality, we obtain
	\begin{align*}
		\|\widehat{f}(\lambda_0\xi)\|_{L^1}&\leqslant\left(\int_{|\xi|\leqslant 1}|\xi|^{-2\alpha_0}\mathrm{d}\xi\right)^{1/2}\|\,|\xi|^{\alpha_0}\widehat{f}(\lambda_0\xi)\|_{L^2}+\left(\int_{|\xi|> 1}|\xi|^{-2\beta_0}\mathrm{d}\xi\right)^{1/2}\|\,|\xi|^{\beta_0}\widehat{f}(\lambda_0\xi)\|_{L^2}\\
		&\lesssim\|\,|\xi|^{\alpha_0}\widehat{f}(\lambda_0\xi)\|_{L^2}+\|\,|\xi|^{\beta_0}\widehat{f}(\lambda_0\xi)\|_{L^2},
	\end{align*}
	where we used the integrabilities of $|\xi|^{-2\alpha_0}\mathrm{d}\xi$ when $|\xi|\leqslant 1$ and $|\xi|^{-2\beta_0}\mathrm{d}\xi$ when $|\xi|>1$ thanks to our condition $2\alpha_0<n<2\beta_0$. According to the scaling argument, one deduces
	\begin{align*}
		\|\widehat{f}(\xi)\|_{L^1}\lesssim \lambda_0^{\frac{n}{2}-\alpha_0}\|\,|\xi|^{\alpha_0}\widehat{f}(\xi)\|_{L^2}+\lambda_0^{\frac{n}{2}-\beta_0}\|\,|\xi|^{\beta_0}\widehat{f}(\xi)\|_{L^2}.
	\end{align*}
	By the balance law, let us select
	\begin{align*}
		\lambda_0=\|\,|\xi|^{\alpha_0}\widehat{f}(\xi)\|_{L^2}^{-\frac{1}{\beta_0-\alpha_0}}\|\,|\xi|^{\beta_0}\widehat{f}(\xi)\|_{L^2}^{\frac{1}{\beta_0-\alpha_0}}
	\end{align*}
	to arrive at
	\begin{align*}
		\|f\|_{L^{\infty}}\leqslant \|\widehat{f}(\xi)\|_{L^1}\lesssim\|\,|\xi|^{\alpha_0}\widehat{f}(\xi)\|_{L^2}^{\frac{2\beta_0-n}{2(\beta_0-\alpha_0)}}\|\,|\xi|^{\beta_0}\widehat{f}(\xi)\|_{L^2}^{\frac{n-2\alpha_0}{2(\beta_0-\alpha_0)}},
	\end{align*}
	which finishes the proof by employing the Plancherel identity.
\end{proof}

\subsection{Estimates related to the initial data}\label{Sub-Section-Est-Data}
\hspace{5mm}According to the definitions of norms for $\ml{Y}_j^{s}(T)$, by using Proposition \ref{Prop-L2-Lm-Real}, it is clear that
\begin{align}\label{Est-Linear}
	\|\partial_tu^{\lin}\|_{\ml{Y}_1^{s}(T)}+\|\partial_tv^{\lin}\|_{\ml{Y}_2^{s}(T)}\lesssim\sum\limits_{(w,j)\in\Lambda_{1,2}^{u,v}}\|(w_0,w_1,w_2)\|_{\mb{D}_{s,m_j}}
\end{align}
for any $T>0$. Besides these linear parts, we also found that the initial data (in the nonlinear forms) appear in \eqref{Nonlinear-Rep-01} and \eqref{Nonlinear-Rep-02}. With the help of Proposition \ref{Prop-L2-Lm-Real} again carrying $w_2\in\{u_1^2-v_1^2,u_1v_1\}$, we obtain the uniform in-time $T$ estimates
\begin{align*}
\|\partial_t\ml{K}_2(t,|D|)(u_1^2-v_1^2)\|_{\ml{Y}_1^{s}(T)}&\lesssim\|u_1^2-v_1^2\|_{H^{s}\cap L^{m_1}},\\
\|\partial_t\ml{K}_2(t,|D|)(u_1v_1)\|_{\ml{Y}_2^{s}(T)}&\lesssim\|u_1v_1\|_{H^{s}\cap L^{m_2}}.
\end{align*}
For this reason, we in the next proposition are going to deal with their right-hand sides, which allows some natural assumptions for the initial data instead of assumptions for the nonlinear data.

\begin{prop}\label{Prop-Data}
Let $m_1,m_2\in[1,2)$ and $s>[\frac{n}{2}-\sigma]_+$. Then, the nonlinear forms of initial data satisfy
\begin{align}
\|u_1^2-v_1^2\|_{H^{s}\cap L^{m_1}}&\lesssim \|u_1\|_{H^{s+\sigma}}^2+\|v_1\|_{H^{s+\sigma}}^2,\label{Est-01}\\
\|u_1v_1\|_{H^{s}\cap L^{m_2}}&\lesssim \|u_1\|_{H^{s+\sigma}}^2+\|v_1\|_{H^{s+\sigma}}^2.\label{Est-02}
\end{align}
\end{prop}
\begin{remark}\label{Rem-m1-m2}
The condition $s>[\frac{n}{2}-\sigma]_+$ ensures
\begin{align*}
\max\{m_1,m_2\}\leqslant\frac{n}{[n-2(s+\sigma)]_+}=+\infty.
\end{align*}
Hence, we do not need to propose an additional assumption for $m_1,m_2$ here.
\end{remark}

\begin{proof}
For any $p\in[1,2]$, by using the fractional Gagliardo-Nirenberg inequality in Proposition \ref{fractionalgagliardonirenbergineq} with $\beta=\frac{n}{s+\sigma}(\frac{1}{2}-\frac{1}{2p})\in[0,1]$, we may get
\begin{align*}
\|u_1^2\|_{L^{p}}=\|u_1\|_{L^{2p}}^2\lesssim\|u_1\|_{L^2}^{2(1-\beta)}\|u_1\|_{\dot{H}^{s+\sigma}}^{2\beta}\lesssim \|u_1\|_{H^{s+\sigma}}^2,
\end{align*}
where the restriction on $\beta$ implies $p\leqslant\frac{n}{n-2(s+\sigma)}$ additionally if $n>2(s+\sigma)$. Similarly, it holds
\begin{align*}
\|v_1^2\|_{L^p}\lesssim \|v_1\|_{H^{s+\sigma}}^2,
\end{align*}
where we pose $p\leqslant\frac{n}{n-2(s+\sigma)}$ if $n>2(s+\sigma)$. By choosing $p\in\{m_1,2\}$, we complete the desired estimate for $\|u_1^2-v_1^2\|_{L^2\cap L^{m_1}}$, in which $s\geqslant\tfrac{n}{4}-\sigma$ from $p=2$ and
\begin{align*}
m_1\leqslant\min\left\{\frac{n}{[n-2(s+\sigma)]_+},\frac{n}{[n-2(s+\sigma)]_+} \right\}\ \mbox{from}\ p=m_1.
\end{align*}
Let us next employ the fractional Sobolev embedding in Proposition \ref{fractionembedd} (where we choose $\alpha_0=0$ and $\beta_0=s+\sigma$) with $s+\sigma>\frac{n}{2}$ to arrive at
\begin{align*}
\|u_1^2\|_{\dot{H}^{s}}+\|v_1^2\|_{\dot{H}^{s}}&\lesssim\|u_1\|_{L^{\infty}}\|u_1\|_{\dot{H}^{s}}+\|v_1\|_{L^{\infty}}\|v_1\|_{\dot{H}^{s}}\\
&\lesssim \|u_1\|_{L^2}^{\frac{2(s+\sigma)-n}{2(s+\sigma)}}\|u_1\|_{\dot{H}^{s+\sigma}}^{\frac{n}{2(s+\sigma)}}\|u_1\|_{\dot{H}^{s}}+\|v_1\|_{L^2}^{\frac{2(s+\sigma)-n}{2(s+\sigma)}}\|v_1\|_{\dot{H}^{s+\sigma}}^{\frac{n}{2(s+\sigma)}}\|v_1\|_{\dot{H}^{s}}\\
&\lesssim\|u_1\|_{H^{s+\sigma}}^2+\|v_1\|_{H^{s+\sigma}}\|v_1\|_{H^{s}}.
\end{align*}
The summary of last estimates completes the proof of \eqref{Est-01}. By following the same way and using $|u_1v_1|\lesssim u_1^2+v_1^2$, we are able to get the analogous estimate \eqref{Est-02}.
\end{proof}

\subsection{Estimates of nonlinear terms}\label{Sub-Section-Est-Nonlinear-Terms}
\hspace{5mm}This part contributes to the estimates for $[\partial_tu(t,\cdot)]^2\pm[\partial_tv(t,\cdot)]^2$ in the $L^p$ and $\dot{H}^{s+\sigma}$ norms, respectively, for any $p\in[1,2]$ (so that we may choose $p\in\{m_1,m_2,2\}$ later) and $s\in\mb{R}_+$. They immediately lead to some a priori estimates for our nonlinearities
\begin{align*}
F_1(t,x):=[\partial_tu(t,x)]^2-[\partial_tv(t,x)]^2\ \mbox{and}\ F_2(t,x):=\partial_tu(t,x)\partial_tv(t,x)
\end{align*}
that will be used later.
\begin{prop}\label{Prop=Estimate-of-Nonlinearities}
Let $p\in[1,2]$ and $s>[\frac{n}{2}-\sigma]_+$.
Then, for any $\ml{V}=(\partial_tu,\partial_tv)^{\mathrm{T}}\in \ml{Y}^{s}(t)$, the nonlinear terms satisfy
\begin{align}
\|\,[\partial_tu(t,\cdot)]^2\pm[\partial_tv(t,\cdot)]^2\|_{L^p}&\lesssim (1+t)^{-\frac{n}{2\sigma}(\frac{2}{\max\{m_1,m_2\}}-\frac{1}{p})}\|\ml{V}\|_{\ml{Y}^{s}(t)}^2,\label{Est-Nonlinear-01}\\
\|\,[\partial_tu(t,\cdot)]^2\pm[\partial_tv(t,\cdot)]^2\|_{\dot{H}^{s+\sigma}}&\lesssim(1+t)^{-\frac{n}{2\sigma}(\frac{2}{\max\{m_1,m_2\}}-\frac{1}{2}+\frac{s+\sigma}{n})}\|\ml{V}\|_{\ml{Y}^{s}(t)}^2.\label{Est-Nonlinear-02}
\end{align}
\end{prop}
\begin{proof}
For one thing, by using the fractional Gagliardo-Nirenberg inequality in Proposition \ref{fractionalgagliardonirenbergineq}, we may derive
\begin{align*}
\mbox{LHS of \eqref{Est-Nonlinear-01}}&\lesssim\sum\limits_{w\in\{u,v\}}\|\partial_tw(t,\cdot)\|_{L^2}^{2(1-\beta)}\|\partial_tw(t,\cdot)\|_{\dot{H}^{s+\sigma}}^{2\beta}\\
&\lesssim (1+t)^{-\frac{n(2-m_1)}{2m_1\sigma}-\frac{n}{2\sigma}(1-\frac{1}{p})}\|\partial_tu\|_{\ml{Y}_1^{s}(t)}^2+(1+t)^{-\frac{n(2-m_2)}{2m_2\sigma}-\frac{n}{2\sigma}(1-\frac{1}{p})}\|\partial_tv\|_{\ml{Y}_2^{s}(t)}^2\\
&\lesssim (1+t)^{-\frac{n}{2\sigma}(\frac{2}{\max\{m_1,m_2\}}-\frac{1}{p})}\|\ml{V}\|_{\ml{Y}^{s}(t)}^2,
\end{align*}
thanks to the condition $p\leqslant\frac{n}{[n-2(s+\sigma)]_+}$ coming from $\beta=\frac{n}{s+\sigma}(\frac{1}{2}-\frac{1}{2p})\in[0,1]$. For another, by using the fractional Sobolev embedding in Proposition \ref{fractionembedd}, we may derive
\begin{align*}
\mbox{LHS of \eqref{Est-Nonlinear-02}}&\lesssim\sum\limits_{w\in\{u,v\}}\|\partial_tw(t,\cdot)\|_{L^{\infty}}\|\partial_tw(t,\cdot)\|_{\dot{H}^{s+\sigma}}\\
&\lesssim\sum\limits_{w\in\{u,v\}}\|\partial_tw(t,\cdot)\|_{L^2}^{\frac{2(s+\sigma)-n}{2(s+\sigma)}}\|\partial_tw(t,\cdot)\|_{\dot{H}^{s+\sigma}}^{\frac{n}{2(s+\sigma)}+1}\\
&\lesssim(1+t)^{-\frac{n(2-m_1)}{2m_1\sigma}-\frac{n+2(s+\sigma)}{4\sigma}}\|\partial_tu\|_{\ml{Y}_1^{s}(t)}^2+(1+t)^{-\frac{n(2-m_2)}{2m_2\sigma}-\frac{n+2(s+\sigma)}{4\sigma}}\|\partial_tv\|_{\ml{Y}_2^{s}(t)}^2\\
&\lesssim(1+t)^{-\frac{n}{2\sigma}(\frac{2}{\max\{m_1,m_2\}}-\frac{1}{2}+\frac{s+\sigma}{n})}\|\ml{V}\|_{\ml{Y}^{s}(t)}^2,
\end{align*}
thanks to the condition $s+\sigma>\frac{n}{2}$. Note that we do not propose an additional assumption on $p$ due to  Remark \ref{Rem-m1-m2}.
\end{proof}

\subsection{Proof of Theorem \ref{Thm-regular-small-data}}\label{Sub-Section-Proof-Regular-Small-Solution}
\hspace{5mm}Let us apply the $(\dot{H}^{s+\sigma}\cap L^{m_j})-\dot{H}^{s+\sigma}$ estimate in $[0,\frac{t}{2}]$ and the $\dot{H}^{s+\sigma}-\dot{H}^{s+\sigma}$ estimate in $[\frac{t}{2},t]$ from Proposition \ref{Prop-L2-inhomogeneous}. Combining Proposition \ref{Prop=Estimate-of-Nonlinearities} to estimate $F_1(\theta,\cdot)$ and $F_2(\theta,\cdot)$ with $p\in\{m_1,m_2,2\}$,
we are able to obtain the following estimate for $\bar{s}\in\{-\sigma,s\}$ and $j\in\{1,2\}$:
\begin{align*}
&(1+t)^{\frac{n(2-m_j)}{4m_j\sigma}+\frac{\bar{s}+\sigma}{2\sigma}}\left\|\int_0^t\partial_t^2\ml{K}_2(t-\theta,|D|)F_j(\theta,\cdot)\mathrm{d}\theta\right\|_{\dot{H}^{\bar{s}+\sigma}}\\
&\lesssim (1+t)^{\frac{n(2-m_j)}{4m_j\sigma}+\frac{\bar{s}+\sigma}{2\sigma}}\int_0^{\frac{t}{2}}(1+t-\theta)^{-\frac{n(2-m_j)}{4m_j\sigma}-\frac{\bar{s}+2\sigma}{2\sigma}}\|F_j(\theta,\cdot)\|_{\dot{H}^{\bar{s}+\sigma}\cap L^{m_j}}\mathrm{d}\theta\\
&\quad+(1+t)^{\frac{n(2-m_j)}{4m_j\sigma}+\frac{\bar{s}+\sigma}{2\sigma}}\int_{\frac{t}{2}}^t(1+t-\theta)^{-\frac{1}{2}}\|F_j(\theta,\cdot)\|_{\dot{H}^{\bar{s}+\sigma}}\mathrm{d}\theta\\
&\lesssim (1+t)^{-\frac{1}{2}}\int_0^{\frac{t}{2}}(1+\theta)^{-\frac{n}{2\sigma}(\frac{2}{\max\{m_1,m_2\}}-\frac{1}{m_j})}\mathrm{d}\theta\|\ml{V}\|_{\ml{Y}^s(T)}^2+(1+t)^{-\frac{n}{2\sigma}(\frac{2}{\max\{m_1,m_2\}}-\frac{1}{m_j})+\frac{1}{2}}\|\ml{V}\|_{\ml{Y}^s(T)}^2,
\end{align*}
where we also considered the asymptotic relations $1+t-\theta\approx 1+t$ if $\theta\in[0,\frac{t}{2}]$ and $1+\theta\approx 1+t$ if $\theta\in[\frac{t}{2},t]$. Let us recall our condition, namely,
\begin{align*}
-\tfrac{n}{2\sigma}\left(\tfrac{2}{\max\{m_1,m_2\}}-\tfrac{1}{m_j}\right)+\tfrac{1}{2}\leqslant0\ \mbox{for all}\ j\in\{1,2\}.
\end{align*}
It yields 
\begin{align*}
\sup\limits_{t\in[0,T]}\left(\sum\limits_{j\in\{1,2\}}(1+t)^{\frac{n(2-m_j)}{4m_j\sigma}+\frac{\bar{s}+\sigma}{2\sigma}}\left\|\int_0^t\partial_t^2\ml{K}_2(t-\theta,|D|)F_j(\theta,\cdot)\mathrm{d}\theta\right\|_{\dot{H}^{\bar{s}+\sigma}}\right)\lesssim\|\ml{V}\|_{\ml{Y}^s(T)}^2
\end{align*}
for $\bar{s}\in\{-\sigma,s\}$. Combining with the estimates in Proposition \ref{Prop-Data} in the representations \eqref{Nonlinear-Rep-01} and \eqref{Nonlinear-Rep-02}, one gets
\begin{align*}
\|\partial_tu^{\nlin}\|_{\ml{Y}^s_{1}(T)}+\|\partial_tv^{\nlin}\|_{\ml{Y}^s_{2}(T)}\lesssim\|\ml{V}\|_{\ml{Y}^s(T)}^2,
\end{align*}
which turns to our desired estimate \eqref{Crucial-03} via the estimate \eqref{Est-Linear} of linear parts. Analogously, by repeating the last steps and \cite[Pages 27-28]{Chen-Takeda=2023}, with the aid of the fractional Leibniz rule in Proposition \ref{fractionleibnizrule}, another crucial estimate \eqref{Crucial-04} also can be completed.


\section*{Acknowledgments}
Wenhui Chen is supported in part by the National Natural Science Foundation of China (grant No. 12301270), Guangdong Basic and Applied Basic Research Foundation (grant No. 2025A1515010240). The author thanks Alessandro Palmieri (University of Bari) for some suggestions on Proposition \ref{fractionembedd}.

\end{document}